\documentclass{amsart}
\usepackage[margin = 1.3 in]{geometry}

\usepackage{amsmath, amsfonts, amssymb, amsthm}
\usepackage{enumerate}
\usepackage{graphicx}
\usepackage{xcolor}
\usepackage{soul}
\usepackage{url}
\usepackage{cite}
\usepackage[colorlinks,citecolor=blue]{hyperref}
\usepackage{verbatim}

\newtheorem{theorem}{Theorem}[section]
\newtheorem{lemma}[theorem]{Lemma}
\newtheorem{proposition}[theorem]{Proposition}
\newtheorem*{theorem:main}{Main Theorem}
\newtheorem*{thmA}{Theorem A}
\newtheorem*{thmB}{Theorem B}

\theoremstyle{definition}
\newtheorem{definition}[theorem]{Definition}
\newtheorem{example}[theorem]{Example}

\newtheorem{fact}[theorem]{Facts}

\theoremstyle{remark}
\newtheorem{remark}[theorem]{Remark}

\numberwithin{equation}{section}

\newcommand{\free}{\mathbb{F}} 
\newcommand{\fltr}{\emptyset = G_0 \subset G_1 \subset \dots \subset G_K = G}
\newcommand{\Rfltr}{\emptyset = G_0 \subset G_1 \subset \dots \subset G_r = G}
\newcommand{\rfa}{\mathcal{FF}(\free,\mathcal{A})} 
\newcommand{\oo}{\Phi} 
\newcommand{\ffa}{\mathcal{A}}
\newcommand{\mc}{\mathcal{MC}(\mathbb{F})}
\newcommand{\pmc}{\mathbb{P}\mathcal{MC}(\mathbb{F})}
\newcommand{\rc}{\mathcal{RC}(\mathcal{A})}
\newcommand{\prc}{\mathbb{P}\mathcal{RC}(\mathcal{A})}
\newcommand{\mrc}{\mathcal{MRC}(\mathcal{A})}
\newcommand{\relOut}{\operatorname{Out}(\mathbb{F}, \ffa)} 
\newcommand{\lamination}{\Lambda^+_{\oo}} 
\newcommand{\PFevalue}{\lambda_{\oo}}
\newcommand{\StableCurrent}{\eta^+_{\oo}}
\newcommand{\UnstableCurrent}{\eta^-_{\oo}}
\newcommand{\RelativeBasis}{\mathfrak{B}_{\ffa}} 
\newcommand{\FreeF}{\mathcal{FF}_n}

\newcommand{\ovrl}{\overline}
\newcommand{\op}{\operatorname}
\newcommand{\alert}{\textcolor{red}}
\newcommand{\la}{\langle}
\newcommand{\ra}{\rangle}
\newcommand{\p}{\prime}
\newcommand{\rank}{\mathfrak{n}}
\begin{document}
\title{Relative Currents}
\author{Radhika Gupta}
\address{Technion, Haifa, Israel}
\curraddr{}
\email{radhikagup@technion.ac.il}
\thanks{The author was partially supported by the NSF grant of Mladen Bestvina (DMS-1607236)}
\date{}
\dedicatory{}
\begin{abstract}
In this paper we define currents relative to a free factor system. We prove that a fully irreducible outer automorphism relative to a free factor system acts with uniform north-south dynamics on a subspace of the space of projective relative currents. 
\end{abstract}
\maketitle
\section{Introduction}
The study of the \emph{outer automorphism group} $\op{Out}(\free)$ of a free group $\free$ of rank $n$ is highly motivated by the mapping class group $\op{MCG}(\Sigma)$ of a surface $\Sigma$. The theory of $\op{MCG}(\Sigma)$ has benefited greatly from its action on the curve complex $\mathcal{C}(\Sigma)$ which was proved to be hyperbolic in \cite{MM:CurveComplex}. Analogously, $\op{Out}(\free)$ acts on the free factor complex $\FreeF$ which was proved to be hyperbolic in \cite{BF:FreeFactorComplex}. But sometimes the parallels between the two theories are not straightforward. For instance, consider a mapping class group element acting on $\mathcal{C}(\Sigma)$ with a fixed point, that is, it fixes a curve $\alpha$ on $\Sigma$. We can then look at its action on the curve complex of the subsurface given by the complement of $\alpha$ and understand it by an inductive process. On the other hand, consider an outer automorphism which fixes a free factor $A$ in $\FreeF$. Since the complement of $A$ in $\free$ is not well defined, we cannot pass to the free factor complex of a free group of lower rank. In \cite{HM:RelativeComplex}, Handel and Mosher define a \emph{free factor complex relative to  a free factor system} (also called \emph{relative free factor complex}) which is an $\op{Out}(\free)$-analog of the curve complex of a subsurface. In the same paper, they also prove hyperbolicity of the relative free factor complex for a \emph{non-exceptional} free factor system. 

In this paper we develop the machinery of \emph{currents relative to  a free factor system}. This machinery is then used in \cite{G:Loxodromic} to classify the outer automorphisms that act with positive translation length on a relative free factor complex. 

In \cite{Bonahon}, Bonahon first defined a space of \emph{geodesic currents} for surfaces such that it contains the set of closed curves as a dense set. He studied the embedding of Teichm\"{u}ller space in the space of geodesic currents and recovered Thurston's compactification of Teichm\"{u}ller space. Currents for free groups were first studied by Reiner Martin \cite{Martin} in his thesis. Analogous to geodesic currents, the space of currents for $\free$ contains the set of conjugacy classes of elements of $\free$ as a dense set. 
Currents for free groups have also been studied in \cite{K:FrequencySpace}, \cite{K:Currents}, \cite{KL:IntersectionForm}. 

Let $\ffa$ be a non-trivial free factor system of $\free$. Here a \emph{non-trivial free factor system} is one which is neither $\emptyset$ nor $\{[\free]\}$. In this paper we define a space of \emph{currents relative to  $\ffa$} (also called \emph{relative currents}) such that it contains the conjugacy classes of elements of $\free$ which are not contained in $\ffa$ as a dense set. Let $\partial^2 \free$ be the space of unoriented bi-infinite geodesics in a Cayley graph of $\free$. A \emph{relative current} is a non-negative, additive, $\free$-invariant and flip-invariant function defined on the set of compact open sets of a subspace $\mathbf{Y}$ (Section~\ref{subsec:Definition}) of $\partial^2 \free$, which depends on $\ffa$. The subspace $\mathbf{Y}$ is defined in such a way that the action of $\free$ on $\mathbf{Y}$ is cocompact, which is important for the space of projectivized relative currents to be compact.  

In order to show that the set of certain conjugacy classes of elements of $\free$ is dense in $\prc$, we extend a relative current to a \emph{signed measured current} which is in fact non-negative on words of bounded length (Section~\ref{sec:ExtensionRC}). We then follow the techniques of \cite{Martin}. 

Our main result is a generalization of a theorem in \cite{Martin} (see also \cite{U:HyperbolicIWIP}) which says that an atoroidal fully irreducible outer automorphism acts with uniform north-south dynamics on the space of projectivized currents, $\pmc$. In \cite{Martin} it is also shown that a non-atoroidal fully irreducible outer automorphism acts with uniform north-south dynamics on a proper subspace of $\pmc$, which is given by the closure of primitive conjugacy classes of $\free$ in $\pmc$. In \cite{U:Generalized}, Uyanik shows that a non-atoroidal fully irreducible acts with \emph{generalized} north-south dynamics on $\pmc$. Generalizing Martin's results, we pass to a similarly defined subspace, denoted $\mrc$, of $\prc$.   
     
Let $\relOut$ be the subgroup of $\op{Out}(\free)$ containing outer automorphisms that preserve $\ffa$. After passing to a finite index subgroup of $\relOut$ we can assume that each conjugacy class of free factor $[A]$ in $\ffa$ is invariant under elements of $\relOut$. An outer automorphism $\oo \in \relOut$ is \emph{fully irreducible relative to $\ffa$} if no power of $\oo$ fixes a non-trivial free factor system of $\free$ properly containing $\ffa$.  Let $\zeta(\ffa)$ be the sum of the number of conjugacy classes of free factors in $\ffa$ and the rank of a cofactor of $\ffa$. 

\begin{thmA}\label{T:Main}
Let $\ffa$ be a non-trivial free factor system of $\free$ with $\zeta(\ffa) \geq 3$. Let $\oo \in \relOut$ be fully irreducible relative to $\ffa$. Then $\oo$ acts with uniform north-south dynamics on $\mrc$, that is, there are only two fixed points $[\StableCurrent]$ and $[\UnstableCurrent]$ and any compact set that does not contain $[\UnstableCurrent]([\StableCurrent])$ uniformly converges to $[\StableCurrent]([\UnstableCurrent])$ under $\oo(\oo^{-1})$ iterates respectively.   
\end{thmA}

We use substitution dynamics techniques to understand the stable current $\StableCurrent$ and the unstable current $\UnstableCurrent$. In the absolute case (when $\ffa = \emptyset$), the transition matrix of a fully irreducible outer automorphism is primitive so Perron-Frobenius theory and the techniques in \cite[Chapter 5]{Queffelec} can be used to define the stable and unstable current. Since the transition matrix of a relative fully irreducible outer automorphism is not primitive and the complement of $\ffa$ is not well-defined, some work needs to be done to define the limiting currents. 
Also some care is required to view an outer automorphism as a substitution due to presence of exceptional paths in a completely split train track representative. See Section~\ref{subsec:AppendixSubstitution} for details. We then study the legal and illegal turn structure of a conjugacy class under iteration by $\oo$.  

Our main application of Theorem A is the following result about loxodromic elements in the relative free factor complex $\rfa$. 
\begin{thmB}[{\cite{G:Loxodromic}}]
Let $\ffa$ be a non-exceptional free factor system of $\free$ and let $\oo \in \relOut$. Then $\oo$ acts loxodromically on $\rfa$ if and only if $\oo$ is fully irreducible relative to $\ffa$.   
\end{thmB}

\subsection*{Plan of the paper} In Section 2, we review some basic definitions. In Section 3, we define relative currents. In Section 4, we state the main proposition (from Section~\ref{subsec:AppendixSubstitution}) about calculating frequencies of paths in a completely split train track representative of a relative fully irreducible outer automorphism. In Section 5, the stable and unstable relative currents are defined. In Section 6, we collect some lemmas about the legal and illegal turn structure of a conjugacy class under iterates by a relative fully irreducible outer automorphism. We conclude with the proof of the main theorem in Section 7. In the appendix we talk about substitution dynamics and extending relative currents to signed measured currents. 

\subsection*{Acknowledgments} I am grateful to my Ph.D. advisor Mladen Bestvina for his guidance, support and endless patience throughout this project. I would also like to thank Gilbert Levitt for pointing out an error in an earlier version, Thomas Goller for discussions about substitution dynamics, and Derrick Wigglesworth and Jon Chaika for helpful conversations. I gratefully acknowledge Mathematical Science Research Institute, Berkeley, California, where the final stage of the present work was completed during Fall 2016. I would also like to thank the referee for helpful suggestions.   
\section{Preliminaries}
\subsection{Marked graphs and topological representatives} We review some terminology from \cite{BH:TrainTracks}. 
Identify $\free$ with $\pi_1(\mathcal{R}, \ast)$ where $\mathcal{R}$ is a rose with $\rank$ petals and $\rank$ is the rank of $\free$. A \emph{marked graph} $G$ is a graph of rank $\rank$, all of whose vertices have valence at least two, equipped with a homotopy equivalence $m : \mathcal{R}\to G$ called a marking. The marking determines an identification of $\free$ with $\pi_1(G,m(\ast))$. 

A homotopy equivalence $\phi: G \to G$ induces an outer automorphism of $\pi_1(G)$ and hence an element $\oo$ of $\op{Out}(\free)$. If $\phi$ sends vertices to vertices and the restriction of $\phi$ to edges is an immersion, then we say that $\phi$ is a \emph{topological representative} of $\oo$. 

A \emph{filtration} for a topological representative $\phi: G \to G$ is an increasing sequence of (not necessarily connected) $\phi$-invariant subgraphs $\fltr$. The closure of $G_r \setminus G_{r-1}$, denoted $H_r$ is a subgraph called the \emph{$r^{th}$-stratum}. 

Let $\gamma$ be a reduced path in $G$. Then  $\phi(\gamma)$ is the image of $\gamma$ under the map $\phi$. The tightened image of $\phi(\gamma)$ is denoted by $[\phi(\gamma)]$.  
 
A path $\sigma$ is a \emph{periodic Nielsen path} if the end points of $\sigma$ are fixed and $\phi^k(\sigma)$ is homotopic relative end points to $\sigma$, for some $k\geq 1$; the smallest such $k$ is called the \emph{period} of $\sigma$. When $k=1$ then $\sigma$ is simply called a \emph{Nielsen path}. A periodic Nielsen path is \emph{indivisible}, denoted as INP, if it does not decompose as a concatenation of non-trivial Nielsen subpaths. A path $\sigma$ is a \emph{pre-Nielsen} path if $\phi^k(\sigma)$ is a Nielsen path for some $k > 0$.  

\subsection{Train track maps}
We recall some more definitions from \cite{BH:TrainTracks}. 

Let $G$ be a marked graph. A \emph{turn} in $G$ is a pair of oriented edges of $G$ originating at a common vertex. A turn is non-degenerate if the edges are distinct, it is degenerate otherwise. A \emph{turn $(e_1, e_2)$ is contained in a filtration element} $G_r$ if both $e_1$ and $e_2$ are contained in $G_r$. If $\gamma$ is an edge path given by $e_1 e_2 \ldots e_{m-1} e_m$, then we say that $\gamma$ \emph{contains the turn} $(\overline{e_{i-1}}, e_i)$, where $\ovrl{e_i}$ denotes opposite orientation.   

For $\phi: G \to G$ a topological representative and an edge $e$ in $G$, set $T\phi(e)$ equal to the first oriented edge of the edge path $\phi(e)$. Given a turn $(e_1, e_2)$, define $T\phi(e_1, e_2) = (T\phi(e_1), T\phi(e_2))$.  
A turn is called \emph{illegal} if under some iterate of $T\phi$ the turn maps to a degenerate turn, it is \emph{legal} otherwise. A path $\gamma$ is called \emph{r-legal} if all of its illegal turns are contained in $G_{r-1}$. 

A matrix called \emph{transition matrix}, denoted $M_r$, is associated to each stratum $H_r$. The $ij^{th}$ entry of $M_r$ is the number of occurrences of the $i^{th}$ edge of $H_r$ in either orientation in the image of the $j^{th}$ edge under $\phi$. A non-negative matrix $M$ is called  \emph{irreducible} if for every $i, j$ there exists $k(i,j)>$ such that $ij^{th}$ entry of $M^k$ is positive. A matrix is called \emph{primitive} or \emph{aperiodic} if there exists $k>0$ such that $M^k$ is positive. A stratum is called \emph{zero stratum} if the transition matrix is the zero matrix. If $M_r$ is irreducible, then its Perron-Frobenius eigenvalue $\lambda_r$ is greater than equal to 1. Such a stratum is \emph{exponentially growing (EG)} if $\lambda_r >1$, it is called \emph{non-exponentially growing (NEG)} otherwise.

A topological representative $\phi : G \to G$ of a free group outer automorphism $\oo$ is a \emph{relative train track map} with respect to a filtration $\fltr$ if $G$ has no valence one vertices, if each non-zero stratum has an irreducible matrix and if each exponentially growing stratum satisfies the following conditions: 
\begin{itemize}
 \item If $E$ is an edge in $H_r$, then the first and the last edges in $\phi(E)$ are also in $H_r$. 
\item If $\gamma \in G_{r-1}$ is a non-trivial path with endpoints in $H_r \cap G_{r-1}$, then $[\phi(\gamma)]$ is non-trivial with endpoints in $H_r \cap G_{r-1}$. 
\item For each $r$-legal path $\beta \subset H_r$, $[\phi(\beta)]$ is $r$-legal.  
\end{itemize}

A reduced path $\sigma \subset G$ has \emph{height} $r$ if the highest stratum it crosses is $G_r$. 
\subsection{Completely split train track maps (CT)}
In \cite{FH:RecognitionTheorem}, Feighn and Handel defined completely split train track maps for outer automorphisms, which are better versions of relative train track maps. Instead of giving a complete definition, we list some facts which are used in this paper and then describe a complete splitting. Let $\phi: G \to G$ be a completely split train track map. 
The following facts proved in different papers can be found in \cite[Section 1.5.2]{HM:ResearchAnnouncement}.
\begin{fact}\label{F:Facts about CTs}
\begin{enumerate}
\item Every periodic Nielsen path has period one. 
\item If $H_r$ is an EG stratum, then there is at most one indivisible Nielsen path (INP) in $G_r$ that intersects $H_r$ nontrivially. 
\item If $H_r$ is an EG stratum and if $\rho_r$ is an INP of height $r$, then $\rho_r$ crosses each edge of $H_r$ at least once, the initial oriented edges of $\rho_r$ and $\overline{\rho}_r$ are distinct oriented edges of $H_r$, and:
\begin{enumerate}
\item $\rho_r$ is not closed if and only of it crosses some edge of $H_r$ exactly once and in this case:
\begin{enumerate}
\item at least one end point of $\rho$ is not in $G_{r-1}$. 
\item there does not exist a height $r$ fixed conjugacy class.
\end{enumerate}
\item $\rho_r$ is closed if and only if it crosses each edge of $H_r$ exactly twice, and in this case:
\begin{enumerate}
\item the endpoint of $\rho_r$ is not in $G_{r-1}$. 
\item the only height $r$ fixed conjugacy classes are those represented by $\rho_r$, its inverse and their iterates. 
\end{enumerate}
\end{enumerate}
\end{enumerate}
\end{fact}

If $H_r$ is an EG stratum, then a non-trivial path in $G_{r-1}$ with end points in $H_r\cap G_{r-1}$ is called a \emph{connecting path}. If an NEG stratum $H_i$ is a single edge $e_i$ such that $\phi(e_i) = e_i u_i$ for a non-trivial closed Nielsen path $u_i$, then $e_i$ is called a \emph{linear edge}. Let $u_i = w_i^{d_i}$ for some $d_i \neq 0$ where $w_i$ is root-free. If $e_i$ and $e_j$ are distinct linear edges such that $\phi(e_i) = e_i w^{d_i}$ and $\phi(e_j) = e_j w^{d_j}$ where $d_i \neq d_j$ and $d_i, d_j>0$, then a path of the form $e_i w^p \overline{e_j}$, where $p \in \mathbb{Z}$, is called an \emph{exceptional path}. 


A decomposition of a path or a circuit $\sigma$ into subpaths is a called a \emph{splitting} if one can tighten the image of $\sigma$ under $\phi$ by tightening the image of each subpath. In other words, there is no cancellation between images of two adjacent subpaths in the decomposition of $\sigma$. 

Let $e$ be an edge in an irreducible stratum $H_r$ and let $k>0$. A maximal subpath $\sigma$ of $[\phi^k(e)]$ in a zero stratum $H_i$ is said to be \emph{$r$-taken}. A non-trivial path or circuit in $G$ is said to be \emph{completely split} if it has a splitting into subpaths each of which is either a single edge in an irreducible stratum, an indivisible Nielsen path, an exceptional path or a connecting path in a zero stratum $H_i$ that is taken and is maximal in $H_i$.   
 
A relative train track map is \emph{completely split} if for every edge $e$ in each irreducible stratum $\phi(e)$ is completely split and if $\sigma$ is a taken connecting path in a zero stratum, then $[\phi(\sigma)]$ is completely split.

\subsection{Free factor system}

A free factor system of $\free$ is a finite collection of proper free factors of $\free$ of the form $\ffa = \{[A_1], \ldots, [A_k] \}$, $k \geq 0$ such that there exists a free factorization $\free = A_1 \ast \cdots \ast A_k \ast F_N$, where $[\cdot]$ denotes the conjugacy class of a subgroup. The free factor $F_N$ is referred to as the \emph{cofactor} of $\ffa$, keeping in mind that it is not unique, even up to conjugacy. There is a partial ordering $\sqsubset$ on the set of free factor systems given as follows: $\ffa \sqsubset \ffa^{\prime}$ if for every $[A_i] \in \ffa$ there exists $[A_j^{\prime}] \in \ffa^{\prime}$ such that $A_i \subset A_j^{\prime}$ up to conjugation. The free factor systems $\emptyset$ and $\{[\free]\}$ are called \emph{trivial free factor systems}. Define \emph{rank}$\mathbf{(\ffa)}$ to be the sum of the ranks of the conjugacy classes of free factors in $\ffa$ and let $\zeta(\ffa) = k+N$.  
   
The main geometric example of a free factor system is as follows: suppose $G$ is a marked graph and $K$ is a subgraph whose non-contractible connected components are denoted $C_1, \ldots, C_k$. Let $[A_i]$ be the conjugacy class of a free factor of $\free$ determined by $\pi_1(C_i)$. Then $\ffa =\{ [A_1], \ldots, [A_k]\}$ is a free factor system. We say $\ffa$ is \emph{realized by $K$} and denote it by $\mathcal{F}(K)$. 

\subsection{Relative free factor complex}\label{subsec:Relative free factor complex}
Let $\ffa$ be a non-trivial free factor system of $\free$. In \cite{HM:RelativeComplex}, the complex of free factor systems of $\free$ relative to  $\ffa$, denoted $\mathcal{FF}(\free; \ffa)$, is defined to be the geometric realization of the partial ordering $\sqsubset$ restricted to the set of non-trivial free factor systems $\mathcal{B}$ of $\free$ such that $\ffa \sqsubset \mathcal{B}$ and $\ffa \neq \mathcal{B}$. The \emph{exceptional} free factor systems are certain ones for which $\rfa$ is either empty or zero-dimensional. They can be enumerated as follows: 
\begin{itemize}
 \item $\ffa = \{ [A_1], [A_2] \}$ with $\free = A_1 \ast A_2$. In this case $\rfa$ is empty. 
\item $\ffa = \{[A]\}$ with $\free = A \ast \mathbb{Z}$. In this case $\rfa$ is $0$-dimensional. 
\item $\ffa = \{ [A_1], [A_2], [A_3]\}$ with $\free =A_1 \ast A_2\ast A_3$. In this case $\rfa$ is also $0$-dimensional. 
\end{itemize}

\begin{theorem}[\cite{HM:RelativeComplex}]
For any non-exceptional free factor system $\ffa$ of $\free$, the complex $\rfa$ is positive dimensional, connected and hyperbolic. 
\end{theorem}
\subsection{Fully irreducible relative to $\ffa$}
Let $\ffa$ be a non-trivial free factor system. An outer automorphism $\oo \in \relOut$ is called \emph{irreducible relative to} $\ffa$ if there is no non-trivial $\oo$-invariant free factor system that properly contains $\ffa$. If every power of $\oo$ is irreducible relative to $\ffa$, then we say that $\oo$ is \emph{fully irreducible relative to} $\ffa$ (or relative fully irreducible).  

Let $\oo \in \relOut$. Then by Lemma 2.6.7 \cite{BFH:Tits}, there exists a relative train track map for $\oo$, denoted $\phi : G \to G$, and filtration $\Rfltr$ such that $\ffa = \mathcal{F}(G_s)$ for some filtration element $G_s$. If $\oo$ is fully irreducible relative to $\ffa$, then $\ffa = \mathcal{F}(G_{r-1})$ and the top stratum $H_r$ is an EG stratum with Perron-Frobenius eigenvalue $\PFevalue>1$. 
\subsection{Bounded cancellation constant and critical length}\label{subsec:BCC}
\begin{lemma}[\cite{C:BCC}]
Let $G$ be a marked metric graph and let $\phi : G\to G$ be a homotopy equivalence. There exists a constant $BCC(\phi)$, called the bounded cancellation constant, depending only on $\phi$ such that for any path $\rho$ in $G$ obtained by concatenating two paths $\alpha, \beta$, we have  
$$ L(\phi(\rho)) \geq L(\phi(\alpha)) + L(\phi(\beta)) - BCC(\phi)$$ where $L$ is the length function on $G$. 
\end{lemma}

Let $BCC(\phi)$ be the bounded cancellation constant for $\phi : G \to G$, a relative train track representative of a relative fully irreducible outer automorphism $\oo$ with top EG stratum $H_r$. The transition matrix of $H_r$ has a positive eigenvector whose smallest entry is one. For an edge $e_i$ in $H_r$, the eigenvector has an entry $v_i>0$. Assign a metric to $G$ such that each edge $e_i$ in $H_r$ is isometric to an interval of length $v_i$ and all edges in $G_{r-1}$ have length one. Then the $r$-length of edges in $H_r$ gets stretched by the PF eigenvalue $\PFevalue$ under $\phi$. Let $l_r$ denote the $r$-length. Let $\alpha, \beta, \gamma$ be $r$-legal paths in $G$. Let $\alpha.\beta.\gamma$ be the path obtained by concatenating these $r$-legal paths. The only $r$-illegal turns possibly occur at the ends of the segments of $\beta$. Thus if $\PFevalue l_r(\beta) -2BCC(\phi) > l_r(\beta)$, then iterations and tightening of $\alpha.\beta.\gamma$ will produce paths with $r$-length of legal segments corresponding to $\beta$ going to infinity. The constant $\frac{2BCC(\phi)}{\PFevalue-1}$ is called the \emph{critical $r$-length} for $\phi$. 

\section{Relative currents}
The goal of this section is to define the space of currents relative to a free factor system. 

\subsection{Boundary of $\free$}\label{subsec:BoundaryFn}
Given $\free$ and a fixed basis $\mathfrak{B}$ of $\free$, let $\op{Cay}(\free, \mathfrak{B})$ be the Cayley graph of $\free$ with respect to $\mathfrak{B}$. The space of ends of the Cayley graph is called the \emph{boundary of} $\free$, denoted by $\partial \free$. It is homeomorphic to the Cantor set. A one-sided cylinder set determined by a finite path $\gamma$ starting at the base point is the set of all rays starting at the base point that cross $\gamma$. Such cylinder sets form a basis for the topology on $\partial \free$ and are in fact both open and closed.  

Let $\Delta$ denote the diagonal in $\partial \free \times \partial \free$. Let \emph{$\partial^2 \free := (\partial \free \times \partial \free - \Delta) / \mathbb{Z}_2$ be the space of unoriented bi-infinite geodesics} in $\op{Cay}(\free, \mathfrak{B})$. This space is also called the \emph{double boundary} of $\free$. 
Finite paths $\gamma$ in $\op{Cay}(\free, \mathfrak{B})$ determine two-sided cylinder sets, denoted $C(\gamma)$, which form a basis for the topology of $\partial^2 \free$. Two-sided cylinder sets are open and compact and hence closed. Compact open sets are given by finite disjoint union of cylinder sets. Also $\partial^2 \free$ is locally compact but not compact. The action of $\free$ on $\partial^2 \free$ is cocompact. 

Let $\ffa = \{[A_1], \ldots, [A_k]\}$, $k>0$, be a free factor system such that $\zeta(\ffa)\geq 3$. 

\begin{definition}[Relative basis] \label{def:RelativeBasis} Choose representatives $A_1, \ldots, A_k$ of conjugacy classes of free factors in $\ffa$ such that $\free = A_1 \ast \ldots \ast A_k \ast F_N$. Let $\RelativeBasis$ be a basis of $\free$ such that a basis of each $A_i$ is a subset of $\RelativeBasis$. Specifically, $$\RelativeBasis = \{a_{11}, \ldots a_{11_s}, \ldots , a_{i1}, \ldots, a_{ii_s}, \ldots, a_{k1}, \ldots, a_{kk_s}, b_1, \ldots, b_p\}$$ where for a fixed $i$, the set $\{a_{ij}\}_{j=1}^{i_s}$ is a basis for $A_i$ and the set $\{b_j\}_{j=1}^{N}$ is a basis for $F_N$. Define a set $B_{\ffa}$ to be the collection of all words $a_{ij}^{\pm}a_{kl}^{\pm}$ of length two such that $i \neq k$ and all the $b_j$s. Note that if $\op{rank}(\ffa) = \op{rank}(\free)$, then the set of $b_i$s is empty. We call $\RelativeBasis$ a \emph{relative basis} of $\free$. 
\end{definition}

\begin{definition}[Double boundary of $\ffa$]\label{def:DBoundaryA} Given a free factor $A$, define $\partial^2 A$ to be the set of unoriented bi-infinite geodesics in $\partial^2 \free$ in the closure of lifts of conjugacy classes of elements in $A$. Then define the double boundary of $\ffa$ as $\partial^2 \ffa := \bigsqcup_{i=1}^k \partial^2 A_i$.      
\end{definition}


Let $\mathbf{Y} = \partial^2 \free \setminus \partial^2 \ffa$. It inherits the subspace topology, denoted $\tau$, from $\partial^2\free$. It can also be given a topology, denoted $\tau'$, where cylinder sets in $\mathbf{Y}$ determined by finite paths that cross at least one word in $B_{\ffa}$ form a basis for the topology. The two topologies are in fact equivalent. Indeed, for every $y \in \mathbf{Y}$ and every basis element $C(\gamma) \cap \mathbf{Y}$ of $\tau$ containing $y$ there is a basis element $C(\gamma')$ of $\tau'$ containing $y$ such that $C(\gamma') \subset C(\gamma) \cap \mathbf{Y}$, where $\gamma$ and $\gamma'$ are finite paths in $\op{Cay}(\free, \RelativeBasis)$, $\gamma \subseteq \gamma' \subset y$ and $\gamma'$ crosses at least one word in $B_{\ffa}$. On the other hand, for every $y \in \mathbf{Y}$, and every basis element $C(\gamma')$ of $\tau'$ containing $y$, the basis element $C(\gamma')\cap \mathbf{Y} \subset \partial^2 \free$ of $\tau$ containing $y$ is such that $ C(\gamma') \cap \mathbf{Y} \subseteq C(\gamma')$.  


\begin{lemma}$\mathbf{Y}$ is locally compact.  \end{lemma} \begin{proof} A space is locally compact if every point has a compact neighborhood. Let $x$ be an element of $\mathbf{Y}$. Take a finite subpath of $x$ that cannot be written as a string of words contained in a single free factor $A$, where $[A] \in \ffa$, and consider the cylinder set determined by that path. Then this cylinder set is a compact open set in $\mathbf{Y}$ containing $x$.  \end{proof}
 
\begin{lemma}\label{L:CocompactAction} The action of $\free$ on $\mathbf{Y}$ is cocompact. \end{lemma} \begin{proof} 
Consider a compact set $C \subset \op{Cay}(\free, \RelativeBasis)$ given by a finite union of cylinder sets determined by all paths with one end point at the origin and with label a word in $B_{\ffa}$. For every bi-infinite  geodesic $\gamma$ in $\mathbf{Y}$ there is a $g \in \free$ such that $g \cdot \gamma$ crosses a path starting at the origin determined by a word in $B_{\ffa}$. 

\end{proof}
\subsection{Definition of relative current}\label{subsec:Definition}
We first recall the definition of a measured current as defined in \cite{Martin}. A \emph{measured current} is an additive, non-negative, $\free$-invariant and flip-invariant function on the set of compact open sets in $\partial^2 \free$. It is uniquely determined by its values on the cylinder sets given by elements in $\free$.

\begin{definition}
With respect to the basis $\RelativeBasis$, let $\free \setminus \ffa$ denote the set of all words in $\free$ that are not contained in any free factor $A_i$, for $1\leq i \leq k$. Note that $\free \setminus \ffa$ contains conjugates of words in $A_i$, as long as the conjugating elements are not in $A_i$.  
\end{definition}

\begin{definition} Let $[\free \setminus \ffa]$ be the set of all conjugacy classes of elements in $\free$ that are not contained in any conjugacy class of a free factor in $\ffa$. Note that an element of $\free \setminus \ffa$ can be contained in the free product of distinct free factors representing elements of $\ffa$. \end{definition}

Let $\mathcal{C}(\mathbf{Y})$ be the collection of compact open sets in $\mathbf{Y}$. A \emph{relative current} is an additive, non-negative, $\free$-invariant and flip-invariant function on $\mathcal{C}(\mathbf{Y})$. Let $\rc$ denote the space of relative currents. A subbasis for the topology of $\rc$ is given by the sets $\{\eta \in \rc: |\eta(C) - \eta_0(C)|\leq \epsilon \}$ where $\eta_0 \in \rc$, $C \in \mathcal{C}(\mathbf{Y})$ and $\epsilon >0$. 

Since $\partial^2 \ffa$ is invariant under the action of $\relOut$, the action of $\op{Out}(\free)$ on $\partial^2 \free$ restricts to the action of $\relOut$ on $\mathbf{Y}$. Thus, $\relOut$ also acts on $\mathcal{C}(\mathbf{Y})$. The group $\relOut$ acts on $\rc$ as follows: let $\eta \in \rc$, $\Psi \in \relOut$ and let $C \in \mathcal{C}(\mathbf{Y})$. Then $$ \Psi.\eta (C) = \eta(\Psi^{-1}(C)).$$

A relative current can also be defined as an $\free$-invariant, locally finite, inner regular measure (called Radon measure) on the $\sigma$-algebra of Borel sets of $\mathbf{Y}$. 

\begin{lemma}
A non-negative, additive function on $\mathcal{C}(\mathbf{Y})$ corresponds to a Radon measure on the Borel $\sigma$-algebra of $\mathbf{Y}$. \end{lemma}
\begin{proof}
Given a non-negative, additive function $\eta$ on $\mathcal{C}(\mathbf{Y})$, define an outer measure $\eta^{\ast}: 2^{\mathbf{Y}} \to [0, \infty]$ as follows: for $A \in 2^{\mathbf{Y}}$ $$ \eta^{\ast}(A):= \op{inf}\left \{ \sum_{i=1}^{\infty} \eta(C_i): A \subseteq \bigcup_{i=1}^{\infty} C_i \text{ where } C_i \in \mathcal{C}(\mathbf{Y}) \text{ is a cylinder set }\right \}.$$ Using additivity of $\eta$ and compactness of $C$, we have $\eta^{\ast}(C) = \eta(C)$ for $C \in \mathcal{C}(\mathbf{Y})$. A cylinder set $C$ in $\mathcal{C}(\mathbf{Y})$ is outer measurable, that is, for every $A \in 2^{\mathbf{Y}}$ we have $\eta^{\ast}(A) = \eta^{\ast}(A \cap C^c)+ \eta^{\ast}(A \cap C)$. 
An outer measure is a measure on the $\sigma$-algebra of outer measurable sets which in this case is the same as the $\sigma$-algebra of Borel sets. Therefore the outer measure $\eta^{\ast}$ is a measure on the Borel $\sigma$-algebra of $\mathbf{Y}$. 
The space $\mathbf{Y}$ is locally compact and Hausdorff and every open set in $\mathbf{Y}$ is $\sigma$-compact, that is, a countable union of compact sets. Also $\eta^{\ast}$ is a non-negative Borel measure on $\mathbf{Y}$ such that it is finite on compact sets. Therefore, $\eta^{\ast}$ is a regular measure. \end{proof}

Thus the space of relative currents can be given a weak-$\ast$ topology, that is, $\displaystyle{\eta_n \to \eta}$ in $\rc$ iff $\displaystyle{\int_\mathbf{Y} f \, d\eta_n \to \int_\mathbf{Y} f \, d\eta}$ for all compactly supported functions $f$ on $\mathbf{Y}$. Since $\mathbf{Y}$ is a locally compact space, by the result in \cite[Chapter III, Section 1]{Bourbaki} $\rc$ is complete. 

\subsection{Coordinates for $\rc$}
Fix a relative basis $\RelativeBasis$ of $\free$. Given $w \neq 1 \in \free$, consider the unique oriented path, denoted $\gamma_w$, determined by $w$ starting at the base point and let $C(w)$ be a shorthand for $C(\gamma_w)$. Note that this cylinder set contains \emph{unoriented} bi-infinite geodesics that cross $\gamma_w$. For $w \in \free \setminus \ffa$, we have $C(w) \subset \mathcal{C}(\mathbf{Y})$. Orbits of cylinder sets of the form $C(w)$ under deck transformations cover $\mathbf{Y}$. We denote $\eta$ applied to $C(w)$ by $\eta(w)$. 

\begin{itemize}
\item Let $v \in \free$. Then $v \cdot C(w)$ is the set of all bi-infinite geodesics that cross an edge path labeled by $w$ starting at the vertex labeled $v$ in the Cayley graph. By $\free$-invariance of a relative current we have that $\eta(C(w)) = \eta(v \cdot C(w))$. Thus we can work just with the cylinder sets determined by finite paths starting at the base point. Since every compact open set is a finite disjoint union of cylinder sets, a relative current is uniquely determined by its values on (cylinder sets determined by) words in $\free \setminus \ffa$.

\item Since a relative current is uniquely determined by its values on $\free \setminus \ffa$, a sequence of relative currents $\eta_n$ converges to $\eta$ if and only if $\eta_n(w) \to \eta(w)$ for all $w \in \free \setminus \ffa$. 

\item For any finite path $\gamma$ in $\op{Cay}(\free,\RelativeBasis)$ we have $C(\gamma) = C(\overline{\gamma})$, where $\overline{\gamma}$ denotes the opposite orientation on $\gamma$. If $w$ and $\gamma_w$ are as above, then $C(w) = C(\gamma_w)= C(\overline{\gamma_w}) = w\cdot C(\overline{w})$. Thus $\eta(w) = \eta(\overline{w})$.   

\item Let $w = e_0e_1\ldots e_l \in \free \setminus \ffa$ where each $e_i \in \RelativeBasis^{\pm}$. Then $C(w) = \bigcup C(we)$, where the union is taken over all basis elements in $\RelativeBasis$ except $e = \overline{e_l}$. Here $\overline{e}$ denotes the inverse of $e$. Also $C(w) = \bigcup \overline{e}\cdot C(ew)$ where $e$ is any basis element other than $\overline{e_0}$. Then additivity of a relative current can be stated as $$\eta(w) = \sum_{e \neq \overline{e_l}} \eta(we) \hspace{.5cm} \text{ or } \hspace{.5cm} \eta(w) = \sum_{e \neq \overline{e_0}} \eta(ew).$$  

For example, let $\free = \la a, b\ra$ and $\ffa = \{[\la a \ra]\}$, we have 
\begin{align*} 
\eta(b) & =\eta(ba)+\eta(bb)+\eta(b\overline{a}), \\ 
\eta(b) & = \eta(ab)+\eta(bb)+\eta(\overline{a}b)  
\end{align*}

\item Let $v,w \in \free \setminus \ffa$ be such that $v$ is a subword of $w$. Then $\eta(w) \leq \eta(v)$. 
\end{itemize}

\begin{example}[Relative current]
Consider a conjugacy class $\alpha \in [\free \setminus \ffa]$ such that $\alpha$ is not a power of any other conjugacy class in $\free$. Then $\eta_{\alpha}(w)$ is the number of occurrences of $w$ in the cyclic words $\alpha$ and $\overline{\alpha}$. Equivalently, one can also count the number of lifts of $\alpha$ that cross the path $\gamma_w$ in the Cayley graph. We call such currents and their multiples \emph{rational relative currents}. For example, let $\free = \la a, b\ra, \ffa = \{[\la a \ra]\}$ and let $\alpha = aba\overline{b}ab$. Then $ \eta_{\alpha}(b) = 3, \eta_{\alpha}(ba) = 2, \eta_{\alpha}(abab) = 1$ and $\eta_{\alpha}(\overline{b}ab)=1$. \end{example}

Given $w \in \free$, a \emph{length $k$ extension} of $w$ is a word $w^{\prime} = wx_1\ldots x_k$ where $x_i \in \RelativeBasis$, $x_i \neq \overline{x_{i+1}}$ and $x_1$ is not the inverse of the last letter of $w$. 
   
\begin{lemma} Any non-negative function $\eta$ on $\free \setminus \ffa$ invariant under inversion and the action of $\free$, and satisfying the condition $$\eta(w) = \sum_{\substack{\text{length one} \\ \text{extension of w}}} \eta(v)$$ for all $w \in \free \setminus \ffa$ determines a relative current.  
\end{lemma} \begin{proof} A set $C \in \mathcal{C}(\mathbf{Y})$ can be written as a disjoint union of cylinder sets $C(w_1), \ldots C(w_k)$ with $w_i \in \free \setminus \ffa$. Then define $\eta(C) := \sum_{i=1}^k \eta(w_i)$. The value $\eta(C)$ does not depend on the choice of $w_i$.
Thus we have an additive and non-negative function on $\mathcal{C}(\mathbf{Y})$ which is invariant under the action of $\free$. \end{proof}
\subsection{Projectivized relative currents}
Let $\prc$ denote the space of projectivized relative currents. It has quotient topology induced from $\rc$. A sequence of projective currents $[\eta_i]$ converges to $[\eta]$ in $\prc$ iff there exist scaling constants $a_i$ such that relative currents $a_i \eta_i$ converge to $\eta$ in $\rc$. 

\begin{example}
Let $\free = \la a, b \ra$ and let $\ffa =\{[\la a \ra]\}$. Consider the sequence $\eta_{a^nb} \in \rc$. This sequence converges to a relative current $\eta_{\infty}$ which is given by $\eta_{\infty}(a^m ba^n) =1$ for all $m,n \geq 0$ and $\eta_{\infty}(w) =0$ for all other $w \in \free \setminus \ffa$. Whereas in the space of measured currents as defined in \cite{Martin}, the sequence $\eta_{a^nb}/n$ converges to the current $\eta_a$. \end{example}

\begin{lemma}\label{P:Compact} $\prc$ is compact. \end{lemma} \begin{proof}

Consider a sequence of projective relative currents $[\eta_n]$. We have to show that it has a convergent subsequence. Fix a relative basis $\RelativeBasis$ and the associated set $B_{\ffa} = \{u_1, \ldots, u_r\}$ (see Definition~\ref{def:RelativeBasis}). Let $\eta_n$ be a representative of $[\eta_n]$ normalized such that $\eta_n(u_i) \leq 1$ for all $u_i \in B_{\ffa}$ and $\eta_n(u_j) = 1$ for some $u_j \in B_{\ffa}$. We have $\eta_n(w) \leq \eta_n(u_i)$ where $w \in \free \setminus \ffa$ and crosses a path labeled $u_i \in B_{\ffa}$ in $\op{Cay}(\free,\RelativeBasis)$. The bounded sequence $\{(\eta_n(u_1), \ldots, \eta_n(u_r))\}_{n \in \mathbb{N}}$ has a subsequence that converges to a non-zero element of $\mathbb{R}^r$. For every $w \in \free \setminus \ffa$, $\{\eta_n(w)\}_{n \in \mathbb{N}}$ is a bounded sequence and hence has a convergent subsequence. Now using the diagonal argument conclude that $\{(\eta_n(w))_{w \in \free \setminus \ffa}\}_{n \in \mathbb{N}}$ has a subsequence that converges to a non-zero element. Thus $\{[\eta_n]\}_{n \in \mathbb{N}}$ has a convergent subsequence in $\prc$.   
\end{proof}
\subsection{Density of rational relative currents}\label{subsec:Density}
\begin{proposition} \label{P:Dense} The set of projectivized relative currents induced by conjugacy classes $\alpha \in [\free \setminus \ffa]$ are dense in $\prc$. \end{proposition}

Let $\RelativeBasis$ be a relative basis of $\free$ and let $|w|$ denote the word length of $w \in \free$ with respect to $\RelativeBasis$. In the absolute case, the following lemma is the main step to prove density of rational currents. But it doesn't directly apply to the relative setting as explained below. 

\begin{lemma}[{\cite[Lemma 15]{Martin}}]
Let $\eta$ be a measured current and let $k \geq 2$. Let $P = 2\rank(2\rank-1)^{2\rank(2\rank-1)^{k-2}}$ be a constant. If $\eta(w_0) \geq P$ for some $w_0 \in \free$ with $|w_0| = k$, then there exists a conjugacy class $\alpha \in [\free]$ and the corresponding measured current $\eta_{\alpha}$ with $\eta(w) \geq \eta_{\alpha}(w)$ for all $w \in \free$ and $|w|\leq k$. \end{lemma}

The proof of the above lemma relies on finding cycles in a certain labeled directed graph associated to $\eta$ defined as follows: vertices are given by words of length ${k-1}$ and edges are given by words of length $k$. A directed edge $w$ joins vertex $u$ to vertex $v$ if $u$ is the prefix of $w$ and $v$ is the suffix of $w$. An edge $w$ is labeled by $\eta(w)$. Since $\eta$ satisfies additivity laws for all words in $\free$, this graph satisfies Kirchhoff's law at each vertex which is crucial to find cycles (which correspond to $\alpha$) in the graph. The same graph defined for a relative current $\eta_0$ does not satisfy Kirchhoff's law at vertices which correspond to words in some free factor $A_i$ for $[A_i] \in \ffa$ because $\eta_0$ is not defined for words in $\ffa$.   

A \emph{signed measured current} on $\partial^2 \free$ is an $\free$-invariant and additive function on the set of compact open sets of $\partial^2\free$. We now restate the above lemma for a signed measured current which is non-negative on words in $\free$ of bounded length.

\begin{lemma}\label{L:Approximating currents from below}
Let $k\geq 2$ and let $\eta$ be a signed measured current such that $\eta(w) \geq 0$ for all $w \in \free$ with $|w|\leq k$. Let $P = 2\rank(2\rank-1)^{2\rank(2\rank-1)^{k-2}}$ be a constant. If $\eta(w_0) \geq P$ for some $w_0 \in \free$ with $|w_0| = k$, then there exists a conjugacy class $\alpha \in [\free]$ and the corresponding measured current $\eta_{\alpha}$ with $\eta(w) \geq \eta_{\alpha}(w)$ for all $w \in \free$ and $|w|\leq k$. 
\end{lemma}
For $\eta_0 \in \rc$, let $\eta$ be a signed measured current such that $\eta(w) = \eta_0(w)$ for $w \in \free \setminus \ffa$ and $\eta(w) \geq 0$ for all words $w \in \free$ with $|w|\leq k$. We call such an $\eta$ a \emph{ $k$-extension} of $\eta_0$. 

\begin{lemma}\label{L:k-extension}
Let $\eta_0$ be a relative current and let $k\geq1$. Then there exists a signed measured current $\eta$ which is a $k$-extension of $\eta_0$. 
\end{lemma} 

To prove the above lemma, we start by defining $\eta$ on length one words in $\ffa$ arbitrarily and then extending the current to length two words by imposing the additivity property. It needs to be checked that the constraints obtained from the additive property are consistent. A detailed proof is given in Section~\ref{sec:ExtensionRC}. Assuming the above lemma is true we now prove Proposition~\ref{P:Dense}.

\begin{proof}[Proof of Proposition~\ref{P:Dense}]We follow the same method of proof as in \cite{Martin}. 
Let $\eta_0$ be a relative current and let $k \geq 2$. Choose $R>0$ such that $R \eta_0(w_0) \geq P$ for some $w_0 \in \free \setminus \ffa$ with $|w_0| = k$. Consider a signed measured current $\eta$ which is a $k$-extension of $\eta_0$. Then by Lemma~\ref{L:Approximating currents from below} applied to $R\eta$, there exists an $\alpha_1 \in \free$ such that $R\eta(w) \geq \eta_{\alpha_1}(w)$ for all $w \in \free$ with $|w|\leq k$.  If $R\eta(w) \leq \eta_{\alpha_1}(w)+P$ for all $w \in \free$ with $|w|\leq k$, then we stop, otherwise we again apply Lemma~\ref{L:Approximating currents from below} to $R \eta - \eta_{\alpha_1}$ to obtain $\alpha_2 \in \free$ such that $R\eta(w) -\eta_{\alpha_1}(w) \geq \eta_{\alpha_2}(w)$ for all $w \in \free$ with $|w|\leq k$. By induction, we have $\sum \eta_{\alpha_i}(w) \leq R\eta(w) \leq \sum \eta_{\alpha_i}(w) + P$ for all words of length less than equal to $k$. 
 
It is necessary that at least one of the $\alpha_i \in [\free \setminus \ffa]$. Indeed, if they were all in $\ffa$, then $\sum \eta_{\alpha_i}(w_0) =0$ which would mean $R\eta(w_0) \leq P$, which is a contradiction. 
 
Now we have that $$\left |\eta(w) - \frac{\sum \eta_{\alpha_i}(w)}{R} \right | \leq \frac{P}{R}$$ for all $w \in \free$ with $|w|\leq k$. For $w \in \free \setminus \ffa$, in fact $$\left |\eta_0(w) - \frac{\sum_{\alpha_i \notin \ffa} \overline{\eta}_{\alpha_i}(w)}{R} \right | \leq \frac{P}{R},$$ where $\overline{\eta}_{\alpha_i}$ is the restriction of $\eta_{\alpha_i}$ to $\mathbf{Y}$. 
 
Since $R$ can be chosen arbitrarily large we can approximate relative currents by sums of rational relative currents for all $w \in \free \setminus \ffa$ with $|w|\leq k$. Now we can approximate $\sum_{\alpha_i \notin \ffa} \eta_{\alpha_i}$ by $\frac{1}{m} \eta_{\beta^m}$ where $\beta^m = w_1^m w_2^m \cdots w_l^m$ and $w_i$ is in the conjugacy class of $\alpha_i$.  
 \end{proof}
\subsection{Relative Whitehead Graph}

\begin{definition} A conjugacy class $\alpha \in [\free \setminus \ffa]$ is $\mathbf{\ffa}$\emph{-separable} if it is contained in a non-trivial free factor system containing $\ffa$. Topologically, $\alpha$ is $\ffa$-separable if there is an $\free$-tree $T$ with set of vertex stabilizers given by $\ffa$ such that an axis of $\alpha$ does not cross every orbit of edges. \end{definition} 
To detect when a conjugacy class is $\ffa$-separable, we use Whitehead's algorithm and a theorem of Stallings \cite{S:WhiteheadAlgorithm}. As defined in \cite{S:WhiteheadAlgorithm}, a collection $\mathcal{C}$ of conjugacy classes in $[\free]$ is \emph{separable} if there exist free factors $F, F^{\p}$ such that $\free = F\ast F^{\p}$ and each conjugacy class in $\mathcal{C}$ is contained in either $F$ or $F^{\p}$. Let $\alpha_i \in [A_i]$, $0 < i \leq k$, be a conjugacy class such that $\alpha_i$ is not contained in any proper free factor of $[A_i]$. We say $\alpha_i$ is \emph{filling} in $[A_i]$. 

\begin{lemma}\label{L:Separability}
A conjugacy class $\alpha \in [\free\setminus \ffa]$ is $\ffa$-separable if and only if the collection of conjugacy classes $\mathcal{C} = \{\alpha, \alpha_1, \ldots, \alpha_k\}$ is separable. \end{lemma}
\begin{proof}
If $\mathcal{C}$ is separable, then there exist a decomposition $\free = F \ast F^{\p}$ such that each conjugacy class in $\mathcal{C}$ is contained either in $F$ or $F^{\p}$. Suppose $\alpha_i \in F$. Then we claim that $A_i$ is contained in $F$ up to conjugation. Suppose not. We have that $F \cap A_i \neq \emptyset$ up to conjugation. Also the intersection of two free factors is a free factor. So $\alpha_i$ is contained in a non-trivial free factor of $A_i$, which is a contradiction. Thus $\{[F], [F^{\p}]\}$ is a non-trivial free factor system containing $\ffa$ that contains the conjugacy class $\alpha$.

On the other hand if $\alpha$ is contained in a proper free factor system $\mathcal{D}$ containing $\ffa$, then $\mathcal{C}$ is separable. 
\end{proof}

\begin{definition}[Whitehead Graph \cite{W:WhiteheadAlgorithm}] Given a basis $\mathfrak{B}$ of $\free$, the Whitehead graph of a collection $\mathcal{C}$ of conjugacy classes, denoted $Wh(\mathcal{C})$, is defined as follows: the vertices are given by basis elements and their inverses. There is an edge connecting vertices $x$ and $y$ if $\overline{x}y$ is a subword of a conjugacy class in $\mathcal{C}$. 
\end{definition}

\begin{theorem}[{\cite[Theorem 4.2]{S:WhiteheadAlgorithm}}]\label{T:Stallings}
Let $\mathcal{C}$ be a collection of conjugacy classes in $[\free]$. If $Wh(\mathcal{C})$ is connected and $\mathcal{C}$ is separable, then there is a cut vertex in $Wh(\mathcal{C})$.  
\end{theorem}

\begin{definition}[Relative Whitehead Graph]
For each $[A_i] \in \ffa$ fix filling conjugacy classes $\alpha_i \in [A_i]$. The relative Whitehead graph of a conjugacy class $\alpha \in [\free \setminus \ffa]$, denoted $Wh(\alpha, \ffa)$, is defined as the Whitehead graph of the collection $\{\alpha, \alpha_1, \ldots, \alpha_k\}$.  
\end{definition}

Note that even though we fix some filling conjugacy classes to define the relative Whitehead graph, detecting $\ffa$-separability of $\alpha$ is independent of them by Lemma~\ref{L:Separability}. 
\subsection{A closed subspace of $\prc$}\label{subsec:A closed subset}

In the absolute case, when a fully irreducible outer automorphism $\Psi$ is a pseudo-Anosov on a surface with one boundary component, the measured current corresponding to the boundary conjugacy class in the space of projectivized measured currents $\mc$ is fixed under the action of $\Psi$. Thus in \cite{Martin}, a closed subspace is considered which is the closure of all primitive conjugacy classes in $\mc$.
For the same reason, we pass to a smaller closed $\relOut$-invariant subspace of $\prc$. Let $$\mrc = \overline{\{ [\eta_\alpha] \in \prc| \alpha \text{ is } \ffa\text{-separable}\}}$$

\begin{lemma}
$[\eta_\alpha] \in \prc$ is in $\mrc$ if and only if $\alpha$ is $\ffa$-separable. \end{lemma}
\begin{proof}

Let's assume that $\alpha$ is not $\ffa$-separable. Then by Theorem~\ref{T:Stallings}, the relative Whitehead graph of $\alpha$ with respect to any relative basis is connected without a cut vertex. Let $w_{\alpha} \in \free \setminus \ffa$ be a cyclically reduced representative of $\alpha$. Consider a relative current $\eta_v$ where $v \in [\free \setminus \ffa]$ such that $\eta_v(w_{\alpha}^2)>0$. This means that any relative Whitehead graph of $v$ contains the Whitehead graph of $\alpha$ as a subgraph and hence is connected without cut vertices. By Theorem~\ref{T:Stallings} and Lemma~\ref{L:Separability}, this implies that $v$ is not $\ffa$-separable. Thus $\eta_v(w_{\alpha}^2) = 0$ for all $\ffa$-separable conjugacy classes $v$ in $[\free \setminus \ffa]$, which in turn implies that $\eta(w_{\alpha}) = 0$ for any $[\eta] \in \mrc$. Since $\eta_\alpha(w_{\alpha}^2)>0$, we have that $\eta_\alpha \notin \mrc$. \end{proof}
\section{Substitution Dynamics}\label{sec:Substitution}
In \cite{Queffelec}, a theory of substitution dynamics is developed for primitive substitutions to study their limiting behavior. This theory can be used to study a fully irreducible outer automorphism by viewing it as a substitution. In Section~\ref{subsec:AppendixSubstitution} of the Appendix, we develop a theory of substitution dynamics for a different class of substitutions in order to study outer automorphisms relative to a free factor system. Here we state the main result (Proposition~\ref{P:TrainTrackSubstitution}) from Section~\ref{subsec:AppendixSubstitution}.

For $\gamma$ and $\alpha$ two paths in a graph $G$, let $(\gamma, \alpha)$ be the number of occurrences of $\gamma$ in $\alpha$. 

\begin{proposition}\label{P:Train track frequency}
Let $\phi: G \to G$ be a completely split train track map. Let $a$ be an edge in an EG stratum $H_r$ such that $\phi(a)$ starts with $a$, and let $\rho_a := \lim_{n \to \infty}\phi^n(a)$. Let $\gamma$ be a path in $G_r$ that crosses $H_r$. Then $$\lim_{n \to \infty} \frac{(\gamma, \phi^n(a))}{\lambda^n} =:d_{\gamma,a}$$ exists and is non-negative. Here $\lambda$ is the Perron-Frobenius eigenvalue of the aperiodic EG stratum $H_r$. If $b \in H_r$ is another edge, then for every $\gamma$ as above, $$d_{\gamma,b} = \kappa d_{\gamma,a}$$ where $\kappa$ is a constant with $\kappa = \kappa(a, b, \phi|_{H_r})$.
\end{proposition}

In general, it is possible that $\gamma$ grows faster than $\lambda$ due to the presence of subpaths in $G_{r-1}$ that grow faster. The point of the above proposition is to ignore the contribution to the growth of $\gamma$ from the lower stratum but still be able to compute frequencies of paths that cross $H_r$ and are not necessarily completely contained in $H_r$.    
\section{Stable and unstable relative current}
In this section, we define the stable and unstable relative currents associated to a fully irreducible outer automorphism relative to $\ffa$. Before we state the general result let's look at some examples. The three examples that follow illustrate the cases when the growth in the stratum corresponding to $\ffa$ is less than, greater than and equal to the growth in the top EG stratum. 

\begin{example}
Let $F_3 = \la a, b, c \ra$. Let $G$ be the rose on three petals labeled $a,b$ and $c$. Consider an outer automorphism $\oo$ given by a train track representative $\phi: G \to G$, where 
$$\phi(a) = a, \phi(b) = bac, \phi(c) = cbac.$$   
Let $\ffa = \{[\la a \ra]\}$. The transition matrix for $\phi$ is given by 
\begin{center}
\begin{tabular}{r c}
&$b \quad c \quad a$ \\
$M = $ & $\begin{bmatrix} 1 & 1 &0  \\ 1 &2 &0 \\1 &1 &1 \end{bmatrix}.$
\end{tabular} \end{center}
Note that $\oo$ is not fully irreducible relative to $\ffa$ because the free factor system $\{[\la b, ac\ra],[\la a \ra]\}$ is $\oo$-invariant. But it is still instructive to understand the limiting behavior in this simple case. 

Let $\rho_b = \lim_{n \to \infty}\phi^n(b)$ be a ray that is fixed by $\phi$. We can view $\phi$ as a substitution $\zeta$ on the alphabet $\mathbb{A} = \{a, b, c\}$. Let $\mathbb{A}_l$ be the set of words of length $l$ on $\mathbb{A}$ that appear in $\rho_b$. For example, $\mathbb{A}_2 = \{ba, ca, cb, ac\}$. Note that the sets $\mathbb{A}_l$ are independent of the specific choice $b$. We define a substitution $\zeta_l$ on $\mathbb{A}_l$ as follows: let $w \in \mathbb{A}_l$ start with $x \in \mathbb{A}$. Then $\zeta_l(w)$ consists of the ordered list of the first $|\zeta(x)|$ subwords of length $l$ of the word $\zeta(w)$. For example, $\zeta_2(ba) = ba \cdot ac \cdot ca$. Let $M_l$ be the transition matrix of $\zeta_l$ and let $\mathcal{B}_l$ be the transition matrix for $\zeta_l$ restricted to words in $\free \setminus \ffa$. We want to calculate the frequency of occurrences of words in $\free \setminus \ffa$ that appear in $\rho_b$.

Let $w \in \mathbb{A}_l$ and let $\beta$ be a word of length $l$ that starts with $b$. Then $$\lim_{n \to \infty} \frac{(w, \phi^n(b))}{\lambda^n} = \lim_{n \to \infty} \frac{M_l^n(w, \beta)}{\lambda^n} = \lim_{n \to \infty} \frac{\mathcal{B}_l^n(w, \beta)}{\lambda^n}=: d_{w,b}$$  
Here $\lambda$ is the PF-eigenvalue of the top EG stratum. See Section~\ref{subsec:AppendixSubstitution} for detailed explanation. For example, in length one and two we have 
\begin{center}
\begin{tabular}{r c}
& $b \quad c$ \\ $\mathcal{B}_1 =$ & $\begin{bmatrix} 1 & 1  \\ 1 &2 \end{bmatrix},$ 
\end{tabular} 
\begin{tabular}{r c}
& $ba \,\,\,  ca \,\,\,  cb \,\,\, ac$ \\ $\mathcal{B}_2 =$ &  $\begin{bmatrix} 1 & 1 &1 & 0 \\ 1 &1 &0 &0 \\0 &1 &2 &0 \\ 1 &1 &1 &1 \end{bmatrix}.$
\end{tabular} 
\end{center}

We take $\beta = b$ and $\beta = ba$ for length one and length two words respectively. Then 
\begin{flalign*} 
(b, \rho_b) = \frac{(5 - \sqrt{5})}{10}, & \quad (c, \rho_b) = \frac{1}{\sqrt{5}}, \\ 
(ac, \rho_b)  = \frac{1}{\sqrt{5}}, & \quad (ba, \rho_b) = \frac{(5 - \sqrt{5})}{10}, \\  (ca, \rho_b)  = \frac{(-5 +3 \sqrt{5})}{10}, & \quad(cb, \rho_b) = \frac{(5 - \sqrt{5})}{10}.
\end{flalign*}
We get $(b, \rho_b) = (ba, \rho_b)$ and $(c, \rho_b) = (ca, \rho_b)+(cb, \rho_b)$ which indicates that additivity holds. One way to calculate the above numbers is to compute the Jordan decomposition of the matrix $\mathcal{B}_l$. 

\end{example}

\begin{example}
Let $F_4 = \la a, b, c, d \ra$. Let $G$ be the rose on four petals labeled $a,b,c,d$. Consider an outer automorphism $\oo$ given by a train track representative $\phi: G \to G$ by $$\phi(a) = abbab, \phi(b) = bababbab, \phi(c) = cad, \phi(d) = dcad.$$   
Let $\ffa = \{[\la a,b \ra]\}$. 
The transition matrix for $\phi$ is given by 
\begin{center}
\begin{tabular}{r c}
& $c\quad d \quad a \quad b$ \\ $M = $ & $\begin{bmatrix}   
	    1 & 1 & 0 & 0  \\
	    1 & 2 & 0 & 0 \\ 
	    1 & 1 & 2 & 3 \\ 
	    0 & 0 & 3 & 5
	    \end{bmatrix}$
\end{tabular}
\end{center}

Let $\rho_c = \lim_{n \to \infty}\phi^n(c)$. We can view $\phi$ as a substitution on the alphabet $\mathbb{A} = \{a, b, c,d\}$. Let $\mathbb{A}_l$ be the set of words of length $l$ on $\mathbb{A}$ that appear in $\rho_c$. We want to calculate the frequency of occurrences of words, which cross $c$ and $d$, in $\rho_c$. Let $w \in \mathbb{A}_l$ and let $\gamma$ be a word of length $l$ that starts with $c$. Using the same notation as in the previous example we have $$\lim_{n \to \infty} \frac{(w, \phi^n(c))}{\lambda^n} = \lim_{n \to \infty} \frac{M_l^n(w, \gamma)}{\lambda^n} = \lim_{n \to \infty} \frac{\mathcal{B}_l^n(w, \gamma)}{\lambda^n} =: d_{w,c}$$
For example, in length two we have $\mathbb{A}_2 = \{ab, ba, bb, ad, bd, ca, da, dc\}$ and $\mathcal{B}_2 =\{ad, bd, ca, da, dc\}$. We get the matrices  
\begin{center}
\begin{tabular}{r c}
& $b \quad c$ \\ $\mathcal{B}_1 =$ & $\begin{bmatrix} 1 & 1  \\ 1 &2 \end{bmatrix},$ 
\end{tabular} 
\begin{tabular}{r c}
& $ca \,\,\,da  \,\,\,dc \,\,\, ad \,\,\,bd $ \\ $\mathcal{B}_2 =$ & $\begin{bmatrix}
   1 &1 &1 &0 &0\\
   1 &1 &0 &0 &0 \\
   0 &1 &2 &0 &0 \\
   1 &1 &1 &0 &0 \\
   0 &0 &0 &1 &1 
  \end{bmatrix} $\end{tabular} 
\end{center}
and compute the frequencies as in the previous example. 
\end{example}

\begin{example}
Let $F_4 = \la a, b, c, d \ra$. Let $G$ be the rose on four petals labeled $a,b,c,d$. Consider an outer automorphism $\oo$ given by a train track representative $\phi: G \to G$ by $$\phi(a) = ab, \phi(b)= bab, \phi(c) = cad, \phi(d) = dcad.$$  
Let $\ffa = \{[\la a,b \ra]\}$. 
The transition matrix for $\phi$ is given by \begin{center}
\begin{tabular}{r c}
& $c\quad d \quad a \quad b$ \\ $M = $ & $\begin{bmatrix}   
	    1 & 1 & 0 & 0  \\
	    1 & 2 & 0 & 0 \\ 
	    1 & 1 & 1 & 1 \\ 
	    0 & 0 & 1 & 2
	    \end{bmatrix}$
\end{tabular}
\end{center}

Let $\rho_c = \lim_{n \to \infty}\phi^n(c)$. We can view $\phi$ as a substitution on the alphabet $\mathbb{A} = \{a, b, c,d\}$. As before we have $$\lim_{n \to \infty} \frac{(w, \phi^n(c))}{\lambda^n} = \lim_{n \to \infty} \frac{M_l^n(w, \gamma)}{\lambda^n} = \lim_{n \to \infty} \frac{\mathcal{B}_l^n(w, \gamma)}{\lambda^n} =: d_{w,c}$$
where $\lambda$ is the PF-eigenvalue of the top stratum. For length two, we have $\mathbb{A}_2 = \{ab, ba, bb, ad, bd, ca, da, dc\}$ and $\mathcal{B}_2 =\{ad, bd, ca, da, dc\}$. We get the matrices  
\begin{center}
\begin{tabular}{r c}
& $b \quad c$ \\ $\mathcal{B}_1 =$ & $\begin{bmatrix} 1 & 1  \\ 1 &2 \end{bmatrix},$ 
\end{tabular}
\begin{tabular}{r c}
& $ca \,\,\,da  \,\,\,dc \,\,\, ad \,\,\,bd $ \\ $\mathcal{B}_2 =$ & $\begin{bmatrix}
   1 &1 &1 &0 &0\\
   1 &1 &0 &0 &0 \\
   0 &1 &2 &0 &0 \\
   1 &1 &1 &0 &0 \\
   0 &0 &0 &1 &1 
  \end{bmatrix} $\end{tabular}
\end{center}
and compute the frequencies as above. 

In all the above examples, the topological representatives of the outer automorphisms were defined on roses whose universal covers are Cayley graphs. Thus, to associate a relative current to such outer automorphisms, we can use coordinates coming from the respective Cayley graphs. In general, a topological representative of an outer automorphism is defined on some marked graph in outer space. Thus, before we show how to associate limiting currents to relatively fully irreducible outer automorphims, we define coordinates with respect to a marked graph. 

\begin{definition}[Coordinates with respect to a marked graph] Let $G$ be a marked metric graph in Culler-Vogtmann's outer space, such that $G$ has a subgraph $\Gamma$ with $\mathcal{F}(\Gamma) = \ffa$. Let $g: \mathcal{R} \to G$ be the marking of $G$. Here $\mathcal{R}$ is the quotient of $\op{Cay}(\free, \RelativeBasis)$ under the action of $\free$. Let $\widetilde{G}$ be the universal cover of $G$. The map $g$ lifts to an $\free$-equivariant map $\widetilde{g}\colon \op{Cay}(\free, \RelativeBasis)  \to \widetilde{G}$. The map $\widetilde{g}$ identifies $\partial^2 \widetilde{G}$ with $\partial^2 \free$ and $\partial^2 \Gamma$ with $\partial^2 \ffa$. Given an edge-path $v$ in $\widetilde{G}$, let $$C(v):= \{(x,y) \in \partial^2 \free \, | \, v \subset (\widetilde{g}(x),\widetilde{g}(x))\}$$ be a compact open set of $\partial^2 \free$ determined by the subpath $v$ of $\widetilde{G}$. For a relative current $\eta$ and a path $v$ of $\widetilde{G}$ that is not entirely contained in the lift of $\Gamma$, $\eta(v)$ is defined to be equal to $\eta(C(v))$. Since $\eta$ is $\free$-equivariant, we may consider $v$ to be a reduced edge-path in $G$ itself. The collection of compact open sets $C(v)$ for all paths $v$ in $G$ that are not entirely contained in $\Gamma$ contains the cylinder sets determined by words in $\free$ that determine a basis for topology of $\partial^2 \free$. Since a relative current is uniquely determined by its values on elements in $\free\setminus \ffa$, it is also uniquely determined by its values on compact open sets determined by reduced paths $v$ in $G$ that are not entirely contained in $\Gamma$.      
\end{definition}

The next lemma defines a limiting current for a relative fully irreducible outer automorphism. 
\begin{lemma}\label{L:Limit Current}
Let $\phi: G \to G$ be a completely split train track representative of $\oo$, a fully irreducible outer automorphism relative to $\ffa$. Let $a$ be an edge in the top EG stratum $H_r$ such that $\rho_a$ is fixed under $\phi$. Let $v$ be any reduced edge path in $G$ that crosses $H_r$. Let $d_{v,a}$ be the frequency of occurrence of $v$ in $\rho_a$. Then the set of values  $$d_{v,a}+d_{\overline{v},a} =:\eta^a_{\phi}(v)$$ define a unique current $\eta^a_{\phi}$ relative to $\ffa$. That is, 
\begin{enumerate}[(a)]
\item $ \eta^a_{\phi}(v)\geq 0$,
\item $\eta^a_{\phi}(v) = \eta^a_{\phi}(\overline{v})$,
\item $\displaystyle{\eta^a_{\phi}(v) = \sum_{e \in E} \eta^a_{\phi}(ve)}$ where $E$ is the set of edges of $G$ incident at end point of $v$ and $e$ is not equal to the inverse of the terminal edge of $v$.   
\end{enumerate}
For an edge $b \neq a$ in $H_r$, we have that $\eta^b_{\phi} = \kappa \eta^a_{\phi}$ for some constant $\kappa(a, b, \phi|_{H_r})$. Thus for every fully irreducible outer automorphism relative to $\ffa$, we get a unique projective relative current, denoted $[\StableCurrent]=[\eta^a_{\phi}]$.  
\end{lemma}
\begin{proof}
By Proposition~\ref{P:Train track frequency}, we know that the values $d_{v,a}$ exist and are non-negative for all reduced paths $v$ in $G$ that cross $H_r$. The equation $(b)$ holds by definition of $\eta^a_{\phi}(v)$. Proposition~\ref{P:Train track frequency} provides a substitution determined by $\phi$. Applying Proposition~\ref{P:MainResultSubstitution} to this substitution we see that $\eta^a_{\phi}(v)$ satisfies Kirchoff's laws, that is, $(c)$ holds. Since a relative current is uniquely determined by its values on compact open sets in $\partial^2 \free$ determined by reduced paths in $G$ that cross $H_r$, we get a unique relative current $\eta^a_{\phi}$. Again by Proposition~\ref{P:Train track frequency}, we have $\eta^a_{\phi}(v) = \kappa \eta^b_{\phi}(v)$ for all reduced paths $v$ in $G$ that cross $H_r$ and for some constant $\kappa$. Thus the projective class $[\eta^a_{\phi}]=:[\StableCurrent]$ of the relative current $\eta^a_{\phi}$ is also unique.       
\end{proof}

The projective relative current $[\StableCurrent]$ is called the \emph{stable current} for $\oo$. The stable current for $\Phi^{-1}$, denoted $[\UnstableCurrent]$, is called the \emph{unstable current} for $\oo$.  

\section{Goodness}
In \cite{BFH:Laminations}, Bestvina, Feighn and Handel studied the legal structure of conjugacy classes under forward and backward iterates of a train track representative of a fully irreducible outer automorphism. In \cite{Brinkmann}, Brinkmann generalized some of those results to relative train track maps which we use in this section. 

Throughout this section $\oo \in \relOut$ will be a fully irreducible outer automorphism relative to $\ffa$ and $\phi: G \to G$ a completely split train track representative of $\oo$ with filtration $\Rfltr$ such that $\mathcal{F}(G_{r-1}) = \ffa$, and $H_r$ is the top EG stratum with PF eigenvalue $\PFevalue>1$. In this section, we use Facts~\ref{F:Facts about CTs} about completely split train track maps. 

In \cite{Brinkmann}, Brinkmann considers the following metric on $G$: edges in $H_r$ get length according to the PF eigenvector for the transition matrix for $H_r$, such that the smallest length is one and hence edges in $H_r$ get stretched ($r$-length) by $\PFevalue$ under the application of $\phi$. Edges in $G_{r-1}$ get length one. See Section~\ref{subsec:BCC}.

Throughout, we use the same notation for a conjugacy class in $[\free]$ and its representative in $G$ which is taken to be cyclically reduced. For a reduced path $\rho$ in $G$, the tightened image of $\rho$ is denoted by $[\phi(\rho)]$. Define $i_r(\rho)$ to be the number of $r$-illegal turns in $\rho$, $l_r(\rho)$ the $r$-length of $\rho$ and $L_r(\rho)$ the length of the longest $r$-legal segment in $\rho$. Recall from Section~\ref{subsec:BCC} that $\displaystyle{L_r^c = \frac{2BCC(\phi)}{\PFevalue-1}}$ is the critical $r$-length where $BCC(\phi)$ is the bounded cancellation constant. 

Denote by $\rho^{-k}$ a path in $G$ with the property that the tightened image of $\phi^{k}(\rho^{-k})$ is $\rho$. For a subpath $\rho$ of a path $\sigma$, let $[\phi^k(\rho)]_{\sigma}$ denote the maximal subpath of $[\phi^k(\rho)]$ contained in $[\phi^k(\sigma)]$.  

The following proposition is a generalization of \cite[Lemma 2.9]{BFH:Laminations}.
\begin{proposition}[{\cite[Lemma 6.2]{Brinkmann}}]\label{P:Brinkmann}
Let $\phi: G \to G$ be a relative train track map and let $H_r$ be an EG stratum. For every $L>0$, there exists $M(L) > 0$ such that if $\rho$ is a path in $G_r$ that crosses $H_r$, then one of the following holds:
\begin{enumerate}[(a)]
\item $[\phi^M(\rho)]$ contains an $r$-legal segment of $r$-length $>L$. 
\item $[\phi^M(\rho)]$ has fewer $r$-illegal turns. 
\item $\rho$ can be expressed as a concatenation $\tau_1 \rho^{\p} \tau_2$, where $l_r(\tau_1) \leq 2L$, $l_r(\tau_2) \leq 2L$, $i_r(\tau_1) \leq 1, i_r(\tau_2) \leq 1$, and $\rho^{\p}$ splits as a concatenation of pre-Nielsen paths (with one $r$-illegal turn each) and segments in $G_{r-1}$.  
\end{enumerate}\end{proposition}

\begin{lemma}[Backward iterations]\label{L:Backward growth}
Let $\phi: G \to G$ be a completely split train track representative of a fully irreducible outer automorphism relative to $\ffa$. Given some number $L_0>0$, there exists $M>0$, depending only on $L_0$ and $H_r$, such that for any subpath $\rho$ of an $\ffa$-separable conjugacy class $\alpha$ realized in $G_r$ with $1 \leq L_r(\rho) \leq L_0$ and $i_r(\rho) \geq 5$, we have $$\left(\frac{10}{9}\right)^n i_r(\rho) \leq i_r(\rho^{-nM})$$ for all $n>0$. \end{lemma}
\begin{proof}
In \cite[Lemma 6.4]{Brinkmann}, Brinkmann proves the same statement for atoroidal outer automorphisms and for any path in $G_r$. The same proof follows by using Facts~\ref{F:Facts about CTs} about completely split train track representatives. 

Given $L=L_0+L_r^c$, choose $M$ as in Proposition~\ref{P:Brinkmann}. Subdivide the path $\rho$ into subpaths $\rho_1, \ldots, \rho_m, \tau$ such that $i_r(\rho_i) = 5$ and $i_r(\tau) < 5$. Let $\rho^{-M}_i$ be the pre-image of $\rho_i$ under $\phi^M$. Then $\rho^{-M}$ is the concatenation of $\rho^{-M}_i$ and $\tau^{-M}$. We claim that $i_r(\rho^{-M}_i)\geq 6$ for all $i$. Suppose for contradiction that $i_r(\rho^{-M}_i) = 5$ for some $i$. Then by Proposition~\ref{P:Brinkmann}, $\rho^{-M}_i$ splits as a concatenation of at least three pre-Nielsen paths and paths in $G_{r-1}$. By Facts~\ref{F:Facts about CTs}, every Nielsen path has period one and there is at most one INP $\sigma$ of height $r$. If $\sigma$ is not closed, then at least one end-point of $\sigma$ is not contained in $G_{r-1}$. Therefore, we cannot have three Nielsen paths in $\rho^{-M}_i$ separated by paths in $G_{r-1}$. If $\sigma$ is closed, then its end point is not in $G_{r-1}$. Since $\alpha$ is $\ffa$-separable, it cannot have two consecutive occurrences of $\sigma$ in it.  Indeed, since $\sigma$ (which is not contained in $G_{r-1}$) is fixed by $\phi$, it is not $\ffa$-separable. Therefore, its relative Whitehead graph is connected without cut points. If $\alpha$ has two consecutive occurrences of $\sigma$, then its relative Whitehead graph will also be connected without cut points but $\alpha$ is $\ffa$-separable. Therefore, $\rho$ and $\rho^{-M}_i$ cannot have two consecutive occurrences of $\sigma$.     

Thus $i_r(\rho^{-M}) \geq 6 m+ i_r(\tau) \geq (10/9) i_r(\rho)$ and the lemma follows by induction. 
\end{proof}

\begin{lemma}[{\cite[Lemma 6.5]{Brinkmann}}]\label{L:Brinkmann 6.5}
Suppose $H_r$ is an EG stratum. Given some $L>0$, there exists some constant $C>0$ such that for all paths $\rho \subset G_r$ with $1 \leq L_r(\rho) \leq L$ and $i_r(\rho)>0$, we have $$ C^{-1} i_r(\rho) \leq l_r(\rho) \leq C i_r(\rho).$$
\end{lemma}

The notion of goodness was introduced in \cite{Martin} and formalized in \cite{BFH:Laminations}. 
\begin{definition}[Goodness]
Given a loop or a path $\alpha$ in $G_r$ that crosses $H_r$, the \emph{good portion}, denoted $g$, of $\alpha$ is the set of $r$-legal segments that are $r$-distance $L_r^c$ away from $r$-illegal turns. The \emph{bad portion}, denoted $b$, is the part of $\alpha$ which is $r$-distance less than equal to $L_r^c$ from an $r$-illegal turn. The $r$-length of $\alpha$ is equal to the $r$-length of $g$ (denoted $g_r(\alpha)$) plus the $r$-length of $b$ (denoted $b_r(\alpha)$). 
 \emph{Goodness} of $\alpha$ is defined as
$$ \mathfrak{g}(\alpha) = \frac{g_r(\alpha)}{l_r(\alpha)}.$$
\end{definition}


\begin{lemma}\label{L:Monotonicity of goodness}
Let $\delta >0$ and $\epsilon >0$ be given. Then there exists an integer $M = M(\delta, \epsilon)$ such that for any $\ffa$-separable conjugacy class $\alpha$ that crosses $H_r$ with $\mathfrak{g}(\alpha) \geq \delta$, we have $\mathfrak{g}(\phi^m(\alpha)) \geq 1-\epsilon$ for all $m \geq M$. \end{lemma}
The proof of the above lemma which is the same as in the absolute case can be found in \cite[Lemma 3.10]{U:HyperbolicIWIP}

\begin{definition}[Desired growth \cite{Brinkmann}]
Let $\sigma$ be a path in $G$ that crosses an EG stratum $H_r$. We say $\sigma$ has desired growth if there exist $N>0, \lambda>1, \epsilon>0$ and a collection of subpaths $S$ of $\sigma$ such that the following hold:
\begin{enumerate}[(a)]
\item For every integer $n>0$ and for every $\rho \in S$, we have $$ \lambda^nl_r(\rho) \leq \op{max}\{ l_r([\phi^{nN}(\rho)]_{\sigma}), l_r(\gamma)\},$$ where $\gamma$ is a subpath of $\sigma^{-nN}$ such that $[\phi^{nN}(\gamma)]_{\sigma^{-nN}} = \rho$.
\item There is no overlap between distinct paths in $S$.
\item The sum of the lengths of the paths in $S$ is at least $\epsilon l_r(\sigma)$. 
\end{enumerate}
\end{definition}

\begin{lemma}\label{L:Desired growth}
Let $\alpha \in [\free]$ be an $\ffa$-separable conjugacy class that crosses $H_r$. Then $\alpha$ has desired growth either under forward iteration or under backward iteration. \end{lemma}
\begin{proof}
Let $L_0 > L_r^c$ be a constant. There are several cases to consider. 
\begin{enumerate}
\item $\displaystyle{\frac{l_r(\alpha)}{i_r(\alpha)} \geq L_0}$. The proof of \cite[Proposition 7.1 (2)(b)(i)]{Brinkmann} shows that in this case $\alpha$ has desired growth in the forward direction. 
\item $\displaystyle{\frac{l_r(\alpha)}{i_r(\alpha)} < L_0}$.
\begin{enumerate}
\item $i_r(\alpha) \geq 5$. By \cite[Proposition 7.1 (2)(b)(ii)]{Brinkmann} and using Lemma~\ref{L:Backward growth}, ~\ref{L:Brinkmann 6.5} we get desired growth in the backward direction. 
\item $i_r(\alpha) <5$. 
We have that $\alpha$ is $\ffa$-separable and crosses $H_r$ non-trivially. Therefore, $\alpha$ is not fixed and does not have two consecutive occurrences of a closed INP. Since $l_r(\alpha)$ is bounded from above, there are only finitely many possibilities for $\alpha \cap H_r$. Suppose the $r$-length of no segment of $\alpha \cap H_r$ grows under $\phi$. Since there are only finitely many segments of $H_r$ of bounded length, after passing to a power we can assume that a segment $\alpha_i$ of $\alpha\cap H_r$ is fixed under $\phi$. Also the end points of $\alpha_i$ are in $H_r \cap G_{r-1}$. There has to be an illegal turn in $\alpha_i$ otherwise it would grow and in fact it has to be an INP because it persists. But at least one end-point of an INP in $G$ is not in $G_{r-1}$, thus we get a contradiction. Therefore we can pass to a uniform power $M$ such that $\phi^M(\alpha)$ satisfies (1) and hence has desired growth in forward direction. 

\end{enumerate}
\end{enumerate}
It can be seen in Brinkmann's proofs that the numbers $N, \lambda, \epsilon$ do not depend on a specific conjugacy class. 
\end{proof}

Let $\phi^{\prime} : G^{\prime} \to G^{\prime}$ be a completely split train track representative of $\oo^{-1}$. Let $l_{r^{\p}}$, $i_{r^{\p}}$, $L_{r^{\p}}^c$ and $C^{\p}$ be the corresponding notation related to $\phi^{\p}$. There exists a constant $B$ such that for any conjugacy class $\alpha$ we have $$\frac{l_{r^{\p}}(\alpha)}{B} \leq l_r(\alpha) \leq B l_{r^{\p}}(\alpha).$$ Let $\mathfrak{g}^{\prime}$ denote the goodness with respect to the train track structure of $\phi^{\prime}$. 
\begin{lemma}\label{L:GoodnessInBothDirections}
Given $\delta >0$, there exists $M>0$ such that for any $\ffa$-separable conjugacy class $\alpha$ that crosses $H_r$ either 
\begin{itemize}
\item $\mathfrak{g}(\phi^{nM}(\alpha)) \geq \delta$ for all $n \geq 1$ or 
\item $\mathfrak{g}^{\prime}((\phi^{\p})^{nM}(\alpha)) \geq \delta$ for all $n \geq1$. \end{itemize}\end{lemma}
\begin{proof}
Let $L_0 > L_r^c$ be the constant from Lemma~\ref{L:Desired growth}. By the same lemma, there exist $N>0, \lambda>1$ and $\epsilon>0$ such that any $\ffa$-separable conjugacy class that crosses $H_r$ has desired growth. 
There are two cases:
\begin{enumerate}[(a)]
\item Let's first consider the case when $\alpha$ has desired growth in the forward direction. This happens when $l_r(\alpha) \geq L_0 i_r(\alpha)$. For case 2(b) in the proof of Lemma~\ref{L:Desired growth} we pass to a uniform power of $\alpha$ which satisfies $l_r(\alpha) \geq L_0 i_r(\alpha)$. 
Let $S$ be the collection of maximal $r$-legal subpaths of $\alpha$ of $r$-length at least $L_0+1$. Then by the choice of $L_0$ we have for $\rho \in S$, $$l_r(\phi^{nN}(\rho)) \geq \PFevalue^{nN} \frac{1}{L_0+1}l_r(\rho).$$ 
We have that the paths in $S$ account for a definite fraction $\epsilon>0$ of $\alpha$. Now $$g_r(\phi^{nN}(\alpha)) \geq \sum_{\rho \in S} [l_r(\phi^{nN}(\rho))]_{\alpha} \geq \sum_{\rho \in S} \PFevalue^{nN} \frac{1}{L_0+1}[l_r(\rho)]_{\alpha} \geq \PFevalue^{nN} \frac{1}{L_0+1} \epsilon l_r(\alpha).$$
We also have $l_r(\phi^{nN}(\alpha)) \leq \PFevalue^{nN} l_r(\alpha)$. Thus we get $$\mathfrak{g}(\phi^{nN}(\alpha)) \geq \frac{\epsilon}{L_0+1}.$$ 


\item If $\alpha$ has desired growth in the backward direction, then by Lemma~\ref{L:Backward growth} and Lemma~\ref{L:Brinkmann 6.5}, we have $$Bl_{r^{\p}}((\phi^{\p})^{ nN}(\alpha)) \geq l_r(\phi^{-nN}(\alpha)) \geq C^{-1}i_r(\phi^{-nN}(\alpha)) \geq \left( \frac{10}{9} \right)^n \frac{1}{C^2B} l_{r^{\p}}(\alpha).$$
Now the number of $r^{\p}$-illegal turns in $(\phi^{\p})^{nN}(\alpha)$ is bounded above by those in $\alpha$. We have 
$$i_{r^{\p}}((\phi^{\p})^{nN}(\alpha)) \leq i_{r^{\p}}(\alpha) \leq C^{\p}l_{r^{\p}}(\alpha).$$
Also the bad portion of $(\phi^{\p})^{nN}(\alpha)$ is bounded from above by $2L_{r^{\p}}^c i_{r^{\p}}((\phi^{\p})^{nN}(\alpha))$. Thus $$ \mathfrak{g}^{\p}((\phi^{\p})^{nN}(\alpha)) \geq 1 - \frac{2L_{r^{\p}}^{c} C^{\p}{B^2C^2}}{(10/9)^n} \geq 1 - \frac{2L_{r^{\p}}^c C^{\p}{B^2C^2}}{(10/9)}.$$ 
\end{enumerate}
Now by Lemma~\ref{L:Monotonicity of goodness}, we find $M>0$ such that either one of the goodness is greater than $\delta$. 
\end{proof}
\section{North-south dynamics}
We are now ready to prove a north-south dynamic result. Recall $\oo$ is a fully irreducible outer automorphism relative to $\ffa$ and $\phi: G \to G$ is a completely split train track representative of $\oo$. We also have a stable current $[\StableCurrent]$ and an unstable current $[\UnstableCurrent]$ in $\mrc$.  
\begin{proposition}\label{L:U}
Given a neighborhood $U$ of $[\StableCurrent]$ in $\mrc$, there exists $0<\delta<1$ and $M(U)>0$ such that for any $[\eta_{\alpha}]\in \mrc$, with $\mathfrak{g}(\alpha)  > \delta$, we have that $\phi^n([\eta_{\alpha}]) \in U$ for all $n \geq M$. \end{proposition}
The proof of the above lemma is similar to the proof of \cite[Lemma 3.11]{U:HyperbolicIWIP}. 

\begin{lemma}\label{L:UorV}  
Given neighborhoods $U$ and $V$ of $[\StableCurrent]$ and $[\UnstableCurrent]$ in $\mrc$, respectively, there exists $M_1>0$ such that for any $\ffa$-separable conjugacy class $\alpha$ that crosses $H_r$ either $\phi^m([\eta_{\alpha}]) \in U$ or $(\phi^{\p})^{m}([\eta_{\alpha}]) \in V$ for all $m \geq M_1$. 
\end{lemma} 
The proof follows from Lemma~\ref{L:GoodnessInBothDirections} and Lemma~\ref{L:U}. 

\begin{proposition}[{\cite[Proposition 3.4]{LU:HyperbolicAutomorphism}}]\label{L:NS1}
Let $\phi: X \to X$ be a homeomorphism of a compact space $X$ and assume that $X$ is sufficiently separable, for example metrizable. Let $Y \subset X$ be a dense set, and let $\mathcal{P},\mathcal{Q}$ be two distinct $\phi$-invariant points in $X$. Assume the following holds: for every neighborhood $U$ of $\mathcal{P}$ and $V$ of $\mathcal{Q}$, there exists an integer $M_2 \geq 1$ such that for all $m \geq M_2$ and all $y \in Y$ one has either $\phi^m(y) \in U$ or $\phi^{-m}(y) \in V$. Then $\phi^2$ has uniform north-south dynamics from $\mathcal{P}$ to $\mathcal{Q}$.  \end{proposition}

\begin{proposition}[{\cite[Proposition 3.5]{LU:HyperbolicAutomorphism}}]\label{NS2}
Let $\phi: X \to X$ be as in Proposition \ref{L:NS1} with distinct fixed points $\mathcal{P}$ and $\mathcal{Q}$ and assume that some power $\phi^s$ with $ s\geq 1$ has uniform north-south dynamics from $\mathcal{P}$ to $\mathcal{Q}$. Then $\phi$ also has uniform north-south dynamics from $\mathcal{P}$ to $\mathcal{Q}$. 
\end{proposition}

\begin{thmA}
Let $\ffa$ be a non-trivial free factor system of $\free$ such that $\zeta(\ffa) \geq 3$. Let $\oo \in \relOut$ be fully irreducible relative to $\ffa$. Then $\oo$ acts with uniform north-south dynamics on $\mrc$.  
\end{thmA}
\begin{proof}
The proof follows from Lemma~\ref{L:UorV}, Proposition~\ref{L:NS1} and Proposition~\ref{NS2}.
\end{proof}
\section{Appendix}
\subsection{Extension of relative currents}\label{sec:ExtensionRC} In this section we will prove Lemma~\ref{L:k-extension},
which says that given a relative current $\eta_0$ there exists a signed measured current $\eta$ which is a $k$-extension of $\eta_0$. We will first show that $\eta_0$ can be extended to a signed measured current $\eta$ which may or may not be non-negative on all words of length less than equal to $k$. We then show how to modify $\eta$ to get a $k$-extension of $\eta_0$. 
\begin{remark}
Throughout this section we will assume that $\ffa$ has only one conjugacy class of a free factor $[A_0]$. When $\ffa$ has more than one free factor in it then the same process can be repeated for all the free factors independently of each other.  
\end{remark}
\emph{Notation:} 
\begin{itemize}
\item Let $\RelativeBasis$ be a relative basis of $\free$. Let $s$ be the rank of the free factor $A_0$. Denote the generators of $A_0$ by $a_i$, $1 \leq i \leq s$. Also let $A:=\{a_1^{\pm}, \ldots, a_s^{\pm} \}$. 
 \item Let $S_{k}$ be the set of words in $A_0$ of length $k$ with respect to $\RelativeBasis$. Let $\#S_k$ denote the cardinality of $S_{k}$.
\item Let $S_k^0$ be a subset of $S_k$ (chosen once and for all) such that for every $w \in S_k$ exactly one of $w$ or $\ovrl{w}$ appears in $S_k^0$.  
\item The letters $e, x, y, z$ will denote elements of $\RelativeBasis$.   
\item Whenever a forward (backward) extension of a word $w$ by $e \in \RelativeBasis$ is written as $we$ $(ew)$ it is to be understood that $e$ is not the inverse of the last (first) letter of $w$. 
\end{itemize}

For every $k>0$ we will define a signed measured current $\eta$ on words in $A_0$ of length $(k-1)$ and use those values together with the additivity laws satisfied by $\eta$ to define $\eta$ on words of length $k$. To start with words of length one, choose arbitrary values for $\eta(a_i)$ for all $1\leq i\leq s$. By induction assume $\eta(v)$ is defined for all words $v$ of length less than equal to $(k-1)$. 
By additivity, for all $v \in S_{k-1}^0$ the following hold: 
\begin{flalign*}
\eta(v) & = \sum_{e \in A}\eta(ve) + \sum_{e \notin A}\eta_0(ve), \\
\eta(\ovrl{v}) & = \sum_{e \in A}\eta(\ovrl{v}e) + \sum_{e \notin A}\eta_0(\ovrl{v}e). 
\end{flalign*}
Since $\eta$ is invariant under taking inverses, the equation obtained from forward extension of $\ovrl{v}$ is the same as the equation obtained from backward extension of $v$.  

Rearranging the equations to have the unknown terms on left hand side we get
\begin{flalign*}
\sum_{e \in A}\eta(ve) & = \eta(v) - \sum_{e \notin A}\eta_0(ve)  = c_v,\\
\sum_{e \in A}\eta(\ovrl{v}e) & = \eta(\ovrl{v}) - \sum_{e \notin A}\eta_0(\ovrl{v}e)  = c_{\ovrl{v}}.
\end{flalign*}
Thus there are $\#S_{k-1}$ equations in $\#S_{k}^0$ variables and the number of variables are more than the number of equations. Denote this system of equations by $E^1_{k-1}$, that is, equations obtained from one edge extensions of length $(k-1)$ words. Similarly we can look at the system $E^i_{k-i}$.  

Consider the augmented matrix $[M|c]$ for the system of equations $E^1_{k-1}$ with rows labeled by $v \in S_{k-1}$ and columns by $w \in S_k^0$.  If $w = ve$ or $\ovrl{w} = ve$ for some $e \in A$, then $M_{v,w}= 1$, otherwise, $M_{v,w}=0$.. We will denote a row vector of $M$ by $r_v$ corresponding to $v \in S_{k-1}$. We make some observations about the matrix $M$. 
\begin{itemize}
\item Each column has exactly two ones. Indeed, $M_{v,w}$ is 1 exactly when $v$ is a prefix of $w$ or $\ovrl{w}$. 
\item  There are $(2s-1)$ non-zero entries in each row because there are $(2s-1)$ possible extensions of $v$ by $e \in A$. 
\item Any two distinct rows can be same in at most one column. Let $w$ be common to two distinct rows $r_{v_1}$ and $r_{v_2}$. Then 
$$ w = v_1e_1 \text{ or } \ovrl{e_1}\hspace{.1cm}\ovrl{v_1} \text{ \quad and \quad} w = v_2e_2 \text{ or } \ovrl{e_2}\hspace{.1cm}\ovrl{v_2}$$ 
for some $e_1, e_2 \in A$. Then it must be true that $v_1$ begins with $\ovrl{e_2}$ and $v_2$ begins with $\ovrl{e_1}$. Thus $w$ is uniquely determined.  
\end{itemize}

\begin{lemma}\label{consistent} \begin{enumerate}[(a)]
\item For every $i \geq 1$, an equation in the system $E^{i+1}_{k-i-1}$ is a linear combination of equations in the system $E^{i}_{k-i}$. Thus it is sufficient to look at the system $E^1_{k-1}$ to obtain all constraints satisfied by $\eta(w)$ for all $w \in S_{k}^0$. 

\item Let $u \in S_{k-2}$. Then we have $$\sum_{x \in A}r_{xu} = \sum_{x \in A} r_{x \ovrl{u}}.$$

\item The set of relations $\displaystyle{\sum_{x \in A}r_{xu} = \sum_{x \in A} r_{x \ovrl{u}}}$ for every $u\in S_{k-2}$ generate any other relation among the rows of $M$. 

\item We also have $$\displaystyle{\sum_{x \in A}c_{xu} = \sum_{x \in A} c_{x \ovrl{u}}}$$ where $c_v$ is the constant term of the equation determined by $v \in S_{k-1}$.

\item The system of equations $E^1_{k-1}$ is consistent and hence has a solution. Thus we can define $\eta$ on words of length $k$. \end{enumerate}\end{lemma}

\begin{proof}\begin{enumerate}[(a)]
\item Let $u \in S_{k-i-1}$. Then $$\eta(u) = \sum_{x \in A} \eta(ux) + \sum_{x \notin A} \eta(ux).$$ By equations in $E^i_{k-i}$, we have $$ \eta(ux) = \sum_{y \in \free, |y|=i} \eta(uxy).$$
Adding all these equations over $x \in \RelativeBasis$ we get $$ \eta(u) = \sum_{x,y \in \free, |x| = 1, |y|=i} \eta(uxy)= \sum_{z \in \free, |z|=i+1} \eta(uz)$$ 
Thus we recovered an equation in $E^{i+1}_{k-i-1}$ by a combination of equations in $E^i_{k-i}$.

\item For every $x \in A$, $M_{xu,w}\neq 0$ exactly when $w=x\,u\, \overline{y}$ or $w=y\, \overline{u}\, \overline{x}$ for some $y \in A$. Therefore if $M_{xu,w}\neq 0$, then $M_{y \overline{u},w}\neq 0$ for some $y \in A$.  

\item  Consider a minimal relation $R$ given by $\displaystyle{\sum_{v \in S_{k-1}} d_{v} r_{v} = 0}$ where $d_v \in \mathbb{R}$. We can rescale the equation such that coefficient of at least one row, say $r_{xu}$ for some $x \in A$ and $u \in S_{k-2}$, is 1. 

For every $y \in A$ and $w = xu\ovrl{y}$, we have $M_{xu, w} = M_{y \ovrl{u}, w} = 1$. Thus $r_{xu}$ and $r_{y \ovrl{u}}$ share exactly one common entry $w$ and no other row has a non-zero entry in $w$. Thus $d_{y \ovrl{u}} = -1$. Now consider $y \in A$. For any $z \in A$ and $w = y\ovrl{u} z$, we have $M_{y \ovrl{u}, w} = M_{\ovrl{z}u, w} = 1$. Thus $d_{\ovrl{z}u} = 1$. Hence our minimal relation is just $\displaystyle{\sum_{x \in A}r_{xu} - \sum_{y \in A}r_{y \ovrl{u}}=0}.$ 
\item We have \begin{flalign*}
\sum_{x \in A}c_{xu} &= \sum_{x \in A}\eta(xu)-\sum_{x \in A,y \notin A}\eta(xuy) \nonumber \\ &= \eta(u) - \sum_{x \notin A}\eta(xu)-\sum_{x \in A,y \notin A}\eta(xuy) \nonumber \\ &= \eta(u) - \sum_{x \notin A, y \in \RelativeBasis}\eta(xuy)-\sum_{x \in A,y \notin A}\eta(xuy) \\
\text{and similarly} \nonumber \\
\sum_{x \in A}c_{x\ovrl{u}} &= \eta(u) - \sum_{x \notin A, y \in \RelativeBasis}\eta(x\ovrl{u}y)-\sum_{x \in A,y \notin A}\eta(x\ovrl{u}y) \nonumber\\
&= \eta(u) - \sum_{x \notin A, y \in \RelativeBasis}\eta(\ovrl{y} u \ovrl{x})-\sum_{x \in A,y \notin A}\eta(\ovrl{y} u \ovrl{x})
\end{flalign*}
We see that $$ \sum_{x \notin A, y \in \RelativeBasis}\eta(xuy)+\sum_{x \in A,y \notin A}\eta(xuy) = \sum_{x \notin A, y \in \RelativeBasis}\eta(\ovrl{y} u \ovrl{x})+\sum_{x \in A,y \notin A}\eta(\ovrl{y} u \ovrl{x}).$$  Geometrically, we are looking at the same subset of $\partial^2 \free$ as a union of cylinder sets in two different ways. See Figure~\ref{F:Proof of (d)} when $\free = \la a, b, c, d \ra$. 
\begin{figure}[h]
    \centering
\includegraphics[scale = .3]{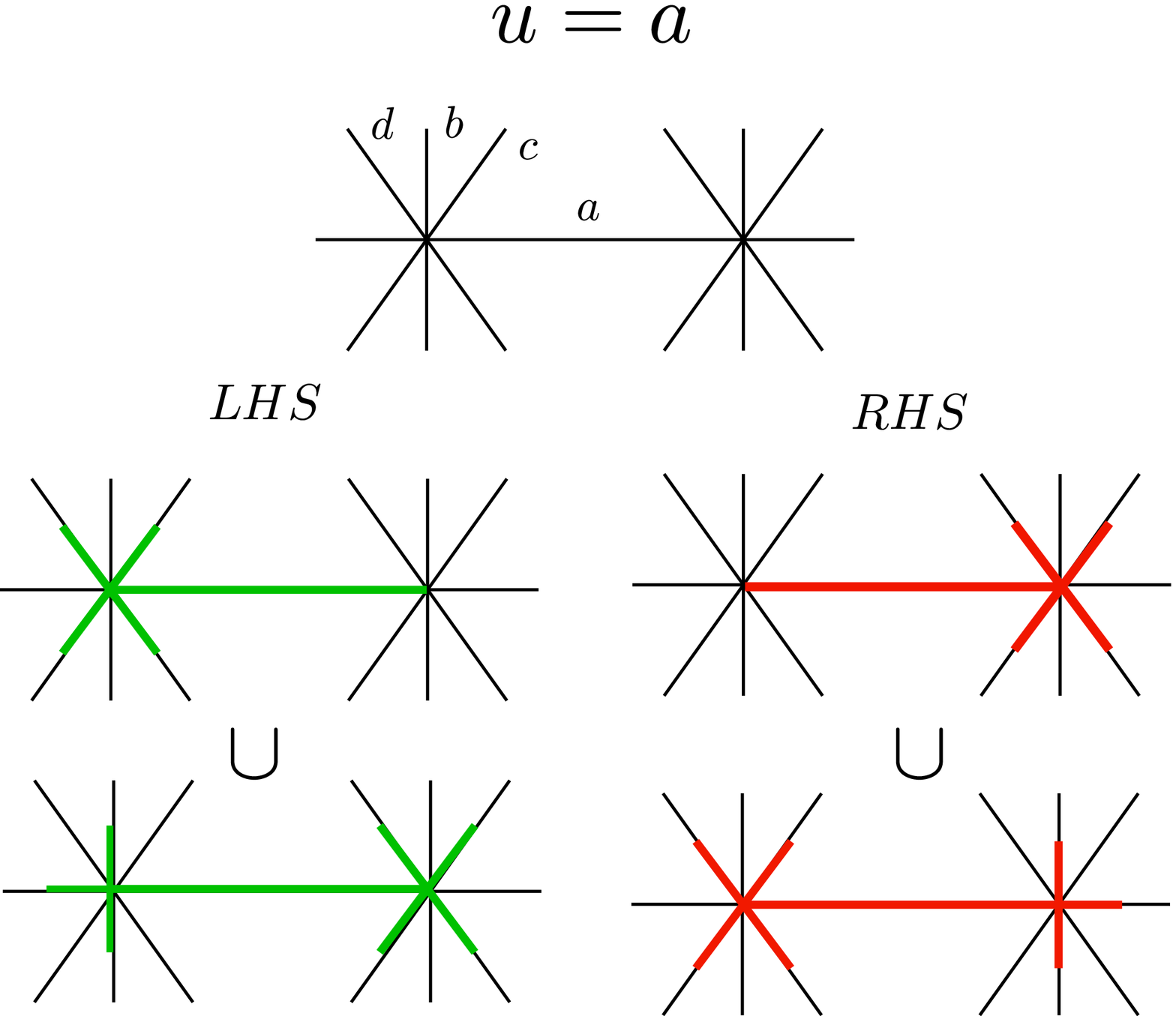}    \caption{}
    \label{F:Proof of (d)}
\end{figure}

\item Since the relations which generate all other relations among the rows of $M$ are consistent, $[M|c]$ has a solution. \end{enumerate} \end{proof}

\begin{proof}[Proof of Lemma~\ref{L:k-extension}]
Given a relative current $\eta_0$, by Lemma \ref{consistent}, we can find a signed measured current $\eta$ such that $\eta_0(w) = \eta(w)$ for all $w \in \free \setminus \ffa$. This extension need not be non-negative on all words of length less than equal to $k$. Let $-M$ for $M>0$ be the smallest value attained by $\eta(w)$ for a word $w \in \ffa$ with $|w|\leq k$. 
Consider a signed measured current $\eta_{\ffa,C}$ defined as follows: $$\eta_{\ffa,C}(w) = \frac{C}{(2s-1)^{|w|-1}} \text{  for } w \in \ffa \text{ and 0 otherwise}.$$
For $C = M(2s-1)^{k-1}$, $\eta+\eta_{\ffa,C}$ is non-negative on words of length less than equal to $k$.   
\end{proof}

\subsection{Substitution Dynamics}\label{subsec:AppendixSubstitution}
Let $\mathbb{A}$ be a finite set with cardinality greater than equal to two. Let $\zeta$ be a substitution on $\mathbb{A}$, that is, a map from $\mathbb{A}$ to the set of non-empty words on $\mathbb{A}$ which associates to a letter $e \in \mathbb{A}$ the word $\zeta(e)$ with length $|\zeta(e)|$. The substitution $\zeta$ induces a map on the set of all words on $\mathbb{A}$ by concatenation, that is, $$\zeta(x_1x_2\ldots x_m) = \zeta(x_1)\zeta(x_2)\ldots \zeta(x_m)$$
where $x_1x_2 \ldots x_m$ is a word on $\mathbb{A}$. Thus we can define iterates $\zeta^n$ for all $n \geq 1$. To the substitution $\zeta$ we associate its transition matrix, denoted $M$, where for $a,b \in \mathbb{A}$, $M(a,b)$ is the number of occurrence of $a$ in $\zeta(b)$. The transition matrix for $\zeta^n$ is given by $M^n$. Likewise, we define a map from $\mathbb{A}^{\mathbb{N}}$ to $\mathbb{A}^{\mathbb{N}}$, the set of all infinite words on $\mathbb{A}$, also denoted $\zeta$, by the formula $\zeta(x_1x_2\ldots) = \zeta(x_1)\zeta(x_2)\ldots$. 

Suppose $\zeta$ admits a fixed point, denoted $\rho \in \mathbb{A}^{\mathbb{N}}$, such that $\zeta^k(\rho) = \rho$ for all $k \geq 1$. From now on we only keep in the alphabet $\mathbb{A}$ the letters that actually appear in $\rho$. 

For every $l >0$, let $\mathbb{A}_l$ denote the set of all words on $\mathbb{A}$ of length $l$ that appear in $\rho$. Define a substitution $\zeta_l$ on $\mathbb{A}_l$ as follows: let $w = x_1 x_2 \ldots x_l \in \mathbb{A}_l$. Define $\zeta_l(w):= w_1 w_2 \ldots w_{|\zeta(x_1)|}$ where $w_i \in \mathbb{A}_l$ and $w_i$ is the length $l$ subword of $\zeta(w)$ starting at the $i^{th}$ position of $\zeta(x_1)$. In other words, $\zeta_l(w)$ consists of the ordered list of the first $|\zeta(x_1)|$ subwords of length $l$ of the word $\zeta(w)$. The substitution $\zeta_l$ extends to a map on the set of all words on $\mathbb{A}_l$. Denote by $|\cdot|_l$ the length of words on $\mathbb{A}_l$. We have $|\zeta_l(w)|_l = |\zeta(x_1)|$.  Denote by $M_l$ the transition matrix for $\zeta_l$. 
It is clear from definitions that $(\zeta^n)_l = (\zeta_l)^n $.


A substitution is called \emph{irreducible} if for every pair $a,b \in \mathbb{A}$ there exists $k:=k(a,b)$ such that $a$ occurs in $\zeta^k(b)$. A substitution is called \emph{primitive} if there exists $k$ such that for every pair $a,b \in \mathbb{A}$, $a$ occurs in $\zeta^k(b)$. 

We are interested in understanding the frequency of occurrence of words on $\mathbb{A}$ that occur in a fixed infinite word $\rho$. In \cite{Queffelec}, a theory for understanding these frequencies for a primitive substitution is developed. We want to generalize the theory of primitive substitutions to substitutions which may not be primitive but are primitive on a subset of the alphabet. 
\subsubsection{Eigenvalues for $M$ and $M_l$}
The main result from this section is Proposition~\ref{P:Eigenvalues}. Consider an alphabet $\mathbb{A} = \bigsqcup_{i = 0}^k B_i$. We define a \emph{partial order on the alphabet} as follows. First define a partial order on subsets of $\mathbb{A}$ given by $B_i > B_j$ for $i<j$. For example, $B_0 > B_1$ and so on. Thus we get a partial ordering on the letters of $\mathbb{A}$ where $a > b$ if $a \in B_i$ and $b \in B_j$ where $i <j$. The alphabet $\mathbb{A}_l$ can now be given a partial lexicographic order as well. 
We will consider a substitution $\zeta$ on $\mathbb{A}$ with the following properties: 
\begin{itemize}
\item For $a \in B_i$, $\zeta(a)$ contains letters only from $B_j$ for $j \geq i$. This implies that the transition matrix $M$ for $\zeta$ is lower triangular block diagonal with respect to the partial order on the set $\{B_i\}_{i=0}^k$. Denote the diagonal blocks of $M$ also by $B_i$ for $0 \leq i \leq k$ where $B_0$ is the top left block, followed by $B_1$ and so on. 
\item If $B_i$ is a primitive block, then $\zeta(a)$ for $a \in B_i$ ends and begins in a letter in $B_i$. 
\item $B_0$ is primitive. 
\end{itemize}

\begin{lemma}\label{L:Substitution main}Let $B_i$ be a primitive block of $M$. After possibly passing to a power of $\zeta$, there exists $a \in B_i$ such that $\zeta(a)$ begins in $a$. 
Also $\rho_a := \lim_{n \to \infty} \zeta^n(a)$ is fixed by $\zeta$, that is, $\zeta(\rho_a) = \rho_a$. If $b \in B_i$ is another letter which begins in $b$ and $\rho_b$ is fixed by $\zeta$, then the set of subwords of $\rho_a$ and $\rho_b$ are the same. \end{lemma}
\begin{proof}
Consider a function $f: B_i \to B_i$ where for $a \in B_i$, $f(a)$ is the first letter of $\zeta(a)$. Since $B_i$ is a finite set, some power of $f$ has a fixed point. After possibly passing to a power, let $a \in B_i$ be a fixed point of $f$. Since $\zeta(a)$ begins with $a$, we have that $\zeta^n(a)$ begins with $\zeta^{n-1}(a)$ for every $n>0$. Thus $\rho_a$ is fixed by $\zeta$. Since $B_i$ is a primitive block, $\zeta^m(a)$ contains $b$ for some $n>0$. Thus subwords that appear in $\rho_b$ also appear in $\rho_a$ and vice versa.  
\end{proof}
We say a word $w$ on $\mathbb{A}$ \emph{crosses} $B_i$ if $w$ contains a letter in $B_i$. We are interested in understanding the frequency of occurrence of words in $\rho:= \rho_0$ which cross $B_0$.   

\begin{example}
Let $\mathbb{A} = \{ a, b, c, d \}$. Let $\zeta$ be given as    
$\zeta(a) = abbab, \zeta(b) = bababbab, \zeta(c) = cad, \zeta(d) = dcad$. 
The transition matrix for $\zeta$ and $\zeta_2$ are given by 
\begin{center}
\begin{tabular}{r c}
& $c\quad d\quad a\quad b$ \\
$M=$ & $\begin{bmatrix}
		1 & 1 & 0 & 0  \\
	    1 & 2 & 0 & 0 \\ 
	    1 & 1 & 2 & 3 \\ 
	    0 & 0 & 3 & 5
	    \end{bmatrix}$, 
\end{tabular}
\begin{tabular}{r c}
& $ca\,\,\, da \,\,\, dc\,\,\,ad \,\,\, bd \,\,\,ab \,\,\,ba \,\,\,bb$ \\
$M_2$ = & $\begin{bmatrix}
   1& 1& 1& 0& 0& 0& 0& 0 \\   
   1& 1& 0& 0& 0& 0& 0& 0 \\
   0& 1& 2& 0& 0& 0& 0& 0 \\
   0& 1& 1& 0& 0& 0& 0& 0 \\
   0& 0& 0& 1& 1& 0& 0& 0 \\
   0& 0& 0& 2& 3& 2& 3& 3 \\
   0& 0& 0& 1& 3& 1& 4& 0 \\
   0& 0& 0& 1& 1& 2& 1& 2 
  \end{bmatrix}$.\end{tabular}
\end{center}
\end{example}

We now want to understand the spectrum of $M_l$. 
\begin{proposition}\label{P:Eigenvalues}
For every $l \geq 2$, the eigenvalues of $M_l$ are those of $M$ with possibly some additional eigenvalues of absolute value less than equal to one.  \end{proposition}

The three lemmas that follow will be used to prove Proposition~\ref{P:Eigenvalues}. 
Since $(\zeta^n)_l = (\zeta_l)^n$, we have $(M^n)_l = (M_l)^n$, which we now denote by $M_l^n$ unless the order needs to be specified.  Denote the rows and columns of $M$ by $R_{x}$ and $C_{x}$ for $x \in \mathbb{A}$, those of $M_l$ by $R_w$ and $C_w$ and those of $M_l^n$ by $R_{n,w}$ and $C_{n,w}$ for $w \in \mathbb{A}_l$. 

\begin{lemma}\label{L:Ml} Let $n\geq 2$. Let $M, M_l, M^n_l$ be transition matrices for $\zeta, \zeta_l, \zeta_l^n$ respectively. Then 
\begin{enumerate}[(a)]
\item $M_l$ is a lower triangular block diagonal matrix with respect to the partial order on $\mathbb{A}_l$.  
\item Let $w \in \mathbb{A}_l$ start with $x \in \mathbb{A}$. Then the sum of the entries of $C_w$ is the same as the sum of the entries of $C_x$ which is equal to $|\zeta(x)|$. 
\item Let $w_1, w_2 \in \mathbb{A}_l$ be such that both words begin with $x \in \mathbb{A}$. Then the entries of $C_{w_1}$ and $C_{w_2}$ differ at most by $(l-1)$. The entries of $C_{n,w_1}$ and $C_{n,w_2}$ also differ at most by $(l-1)$.  
\end{enumerate}\end{lemma}
\begin{proof}\begin{enumerate}[(a)]
\item Clear from definitions of $M$ and $M_l$. 
\item Let $w,x$ be as in the statement of the lemma. Then $|\zeta_l(w)|_l = |\zeta(x)|$, which implies that column sum of $C_w$ is same as that of $C_x$. 
\item Let $w_1, w_2, x$ be as in the statement of the lemma. Then $\zeta_l(w_1)$ and $\zeta_l(w_2)$ differ only when the length $l$ words starting at some position in $\zeta(x)$ are not subwords of $\zeta(x)$. If $|\zeta(x)| \geq l$, then the first time such a word occurs is when it starts at position $(l-1)$ from the end of $\zeta(x)$. If $|\zeta(x)| < l$, then $\zeta_l(w_1)$ and $\zeta_l(w_2)$ can differ in at most $|\zeta(x)| < l$ length $l$ words. Thus there are at most $(l-1)$ such words. Replace $\zeta, \zeta_l$ by $\zeta^n, (\zeta^n)_l$ above to conclude that entries of $C_{n,w_1}$ and $C_{n,w_2}$ also differ at most by $(l-1)$.  \qedhere
\end{enumerate}\end{proof}

\begin{lemma}\label{BoundingEigenvalue}
If $Q$ is a $s \times s$ matrix such that absolute values of all its entries are bounded above by $\delta >0$, then the absolute values of the eigenvalues of $Q$ are bounded above by $s \delta$. \end{lemma}
\begin{proof}
 Let $\lambda \neq 0$ be an eigenvalue of $Q$ and let $v = (v_1, \ldots, v_s)$ be a corresponding eigenvector. Let $r_i$ denote rows of $Q$. Then $|r_i \cdot v| = |\lambda v_i|$ which gives $|\lambda v_i| \leq \delta \sum_{j=1}^s |v_j|$ for every $1 \leq i \leq s$. Adding all the inequalities together we get $|\lambda| \leq s \delta$.   \end{proof}

For every $B_i \subset \mathbb{A}$, let $\widetilde{B_i} \subset \mathbb{A}_l$ be the set of all words $w$ that start with a letter in $B_i$ and such that $w$ does not cross $B_j$ for any $j< i$. For every $B_i \subset \mathbb{A}$, let $\overline{B_i} \subset \mathbb{A}_l$ be the set of all words $w$ that start with a letter in $B_i$ and there exists a $j< i$ such that $w$ crosses $B_j$ (note that $\overline{B_0}$ is empty). Then we have that $\widetilde{B_i} \cup \overline{B_i}$ is the union of all words of length $l$ that start with a letter in $B_i$.  The partial order on $\mathbb{A}_l$ defined earlier gives that $\widetilde{B_0}> \overline{B_1} > \widetilde{B_1} > \ldots > \overline{B_k} > \widetilde{B_k}$. The matrix $M_l$ is lower triangular block diagonal with respect to this partial order on $\mathbb{A}_l$. For a subset $S \subset \mathbb{A}_l$, denote by $S$ the transition matrix of $\zeta_l$ restricted to $S$. 


\begin{lemma}\label{L:B_i}
\begin{enumerate}[(a)]
\item For every $0 \leq i \leq k$, the characteristic polynomial of $B_i$ divides the characteristic polynomial of $\widetilde{B_i}$. 
\item The eigenvalues of $\widetilde{B_i}$ are those of $B_i$ with possibly some additional eigenvalues of absolute value less than equal to one. 
\item The eigenvalues of $\overline{B_i}$ have absolute value less than equal to one. 
\end{enumerate}\end{lemma}
\begin{proof}\begin{enumerate}[(a)]
\item Consider the matrix $P_i = \widetilde{B_i} - \lambda I$. We will do certain row and column operations on this matrix to reduce it to a lower triangular block diagonal matrix with $B_i - \lambda I$ as a diagonal block, which would imply that the characteristic polynomial of $B_i$ divides the characteristic polynomial of $\widetilde{B_i}$. For later use we denote the other diagonal block of $P_i$ by $Q$. 

We first perform the following row operations: for every $x \in B_i$, choose a word $w \in \widetilde{B_i}$ such that $w$ starts with $x$. For every such $w$, replace the row $R_w$ of $\widetilde{B_i}$ by the sum of rows $R_u$ for all $u \in \widetilde{B_i}$ that start with $x$. Rearrange the rows and columns such that top left block is indexed by the chosen words $w$. The rearranged matrix is denoted by $P_i^{\p}$. The top left block of $P_i^{\p}$ is exactly $B_i - \lambda I$. Indeed, suppose $w, u \in \widetilde{B_i}$ in the top left block of $P_i^{\p}$ start with $x,y \in B_i$, respectively. Then $P_{i}^{\p}(w,v)$ is exactly the number of occurrences of $x$ in $\zeta(y)$. 

Now for any two columns $C_{w_1}$ and $C_{w_2}$ of $P_i^{\p}$, where $w_1, w_2$ start with the same letter in $B_i$, the first few entries (as many as the number of rows in the top left block of $P_i^{\p}$) are equal. Now perform column operations as follows: for every $x \in B_i$ and the chosen word $w$ in the top left block, subtract $C_{w}$ from $C_u$ for every $u \neq w$ that start with $x$. Thus we have a lower triangular block diagonal matrix, again denoted $P_i^{\p}$, with  diagonal blocks $B_i - \lambda I$ and $Q$. 

\item Consider the lower block diagonal matrix $P_i^{\p}$ from above. Eigenvalues of $P_i^{\p}$ not coming from the block $B_i-\lambda I$ come from the lower block, denoted $Q$. By Lemma \ref{L:Ml}(c), the entries of $Q$ are bounded in absolute value by $(l-1)$. We claim that the eigenvalues of $Q$ are bounded in absolute value by one.  

Let $\lambda_0$ be an eigenvalue of $Q$ and hence of $\widetilde{B_i}$. Then for $n\geq 1$, $\lambda_0^n$ is an eigenvalue of $(\widetilde{B_i})^n$ which is a diagonal block of $(M_l)^n = (M^n)_l$.   
Thus $\lambda_0^n$ is an eigenvalue of $(\widetilde{B_i})^n$ that does not come from eigenvalue of $B_i^n$, the corresponding diagonal block of $M^n$. Applying part $(a)$ to $\zeta^n$, $(\widetilde{B_i})^n$ can also be put in a lower triangular block diagonal form with diagonal blocks $B_i^n - \lambda I$ and $Q^{\p}$. Since the entries of $Q^{\p}$ are bounded by $(l-1)$,  by Lemma \ref{BoundingEigenvalue}, every eigenvalue of $Q^{\p}$ is bounded in absolute value by size of $Q^{\p}$ times $(l-1)$. Thus $|\lambda_0^n|$ is uniformly bounded which can happen only when $|\lambda_0| \leq 1$.     

Thus all eigenvalues of $\widetilde{B_i}$ are eigenvalues of $B_i$ with the exception of some eigenvalues whose absolute value is less than equal to one. 

\item Let $\lambda$ be an eigenvalue of $\overline{B_i}$. Then $\lambda^n$ is an eigenvalue of $(\overline{B_i})^n$, the diagonal block of $(M^n)_l$ corresponding to words that start with a letter in $B_i$ and there exists a $j <i$ such that they cross $B_j$. For every $n$, the entries of $(\overline{B_i})^n$ are bounded by $(l-1)$. Indeed, if $w$ is a length $l$ word that starts with $x$, then only the words that start at some position less than $l$ away from the last letter of $\zeta^n(x)$ belong to $(\overline{B_i})^n$. This implies that eigenvalues of $(\overline{B_i})^n$ are uniformly bounded. That is, $|\lambda^n|$ is uniformly bounded which can happen only when $|\lambda| \leq 1$. \qedhere    
\end{enumerate}\end{proof}

\begin{proof}[Proof of Proposition~\ref{P:Eigenvalues}]
Since eigenvalues of a lower triangular block diagonal matrix are obtained from eigenvalues of each block the proposition follows from Lemma~\ref{L:B_i}. 
\end{proof}
\subsubsection{Frequency of words}
The main result in this subsection is Proposition~\ref{P:MainResultSubstitution}. Recall that we want to understand the frequency of occurrence of words which cross $B_0$ in $\rho$. Let $\lambda$ be the top eigenvalue of the block $B_0$ of $M$. Consider a subset $\mathcal{B}_l :=  \widetilde{B_0} \cup (\bigcup_{i=1}^k \overline{B_i})$ of $\mathbb{A}_l$. Then the set of all length $l$ words that cross $B_0$ is a subset of $\mathcal{B}_l$. The transition matrix of $\zeta_l$ restricted to $\mathcal{B}_l$ is also lower triangular block diagonal with respect to the order $\widetilde{B_0} > \overline{B_1} > \ldots > \overline{B_k}$ of words in $\mathcal{B}_l$. Then by Lemma~\ref{L:B_i}, $\lambda>1$ is the top eigenvalue of $\mathcal{B}_l$ with multiplicity one. Since $\mathcal{B}_l$ is a diagonal block of $M_l$, we have $M^n_l(w,\alpha) = \mathcal{B}_l^n(w,\alpha)$ for all $w, \alpha \in \mathbb{A}_l$ that cross $B_0$. 

For $w, v$ words on $\mathbb{A}$ or $\mathbb{A}_l$ let $(w,v)$ denote the number of occurrences of $w$ in $v$. 
\begin{lemma}
 Let $a \in B_0$ and let $\rho_a = \lim_{n \to \infty} \zeta^n(a)$ be such that $\zeta(\rho_a) = \rho_a$. Let $w \in \mathbb{A}_l$ be a word that crosses $B_0$. Then $$  \text{frequency of occurrence of $w$ in $\rho_a$} = \lim_{n \to \infty} \frac{(w, \zeta^n(a))}{\lambda^n}=: d_{w,a}$$ exists and is non-negative. Here $\lambda$ is the top eigenvalue of $B_0$. 
\end{lemma}
\begin{proof}
 Let $\alpha \in \mathbb{A}_l$ start with $a$. For $n$ large, the number of occurrences of $w$ in $\zeta^n(a)$ is approximately the same as the number of occurrences of $w$ in $\zeta^n_l(\alpha)$.
Also $$ (w, \zeta_l^n(\alpha)) = M^n_l(w,\alpha). $$ We have
\begin{flalign*}
 \lim_{n \to \infty}\frac{ (w, \zeta^n(a))}{\lambda^n} = \lim_{n \to \infty}\frac{ (w,\zeta^n_l(\alpha))}{\lambda^n} = \lim_{n \to \infty}\frac{M_l^n(w,\alpha)}{\lambda^n} = \lim_{n \to \infty}\frac{\mathcal{B}_l^n(w,\alpha)}{\lambda^n} =:d_{w,a}.
\end{flalign*}

Indeed, the limit exists because $\lambda$ is the top eigenvalue of $\mathcal{B}_l$. 
The limit is non-negative because it is a sequence of non-negative numbers. The limit does not depend on the exact choice of $\alpha$ because by Lemma~\ref{L:Ml}(c), any two columns of $M_l^n$ starting with the same letter in $\mathbb{A}$ differ by a bounded amount and thus give the same limit.   
\end{proof}

\begin{lemma}[Kirchhoff's Law]\label{Kirchhoff'sLaw}
Let $a \in B_0$. Let $w \in \mathbb{A}_l$ cross $B_0$. Let $we$ and $ew$ be length one extensions of $w$ by $e \in \mathbb{A}$. Then $$d_{w,a} = \sum_{e \in \mathbb{A}} d_{we,a} = \sum_{e \in \mathbb{A}} d_{ew,a}.$$ \end{lemma}
\begin{proof}
We have $(w,\zeta^n(a))$ and $\sum_{e \in \mathbb{A}}(we,\zeta^n(a))$ differ only when $\zeta^n(a)$ ends in $w$, so their difference is at most one. Thus 
$$ \left| \frac{(w,\zeta^n(a))}{\lambda^n} - \sum_{e \in \mathbb{A}}\frac{(we,\zeta^n(a))}{\lambda^n} \right| \to 0 \text{ as } n \to \infty$$ which implies that $d_{w,a} = \sum_{e \in \mathbb{A}} d_{we,a}$. Similarly $d_{w,a} = \sum_{e \in \mathbb{A}} d_{ew,a}$. 
\end{proof}

\begin{lemma}
Let $a, b \in B_0$ be distinct. Then $$d_{w,b} = \kappa  d_{w,a}$$ for every word $w$ that crosses $B_0$ where $\kappa = \kappa(a,b,\zeta|_{B_0})$.
\end{lemma}
\begin{proof}
Let's first consider the case when length of $w$ is one. We have $\zeta$ restricted to $B_0$ is primitive with top eigenvalue $\lambda>1$. Then $$d_{w,a} = \lim_{n \to \infty}\frac{M^n(w,a)}{\lambda^n} = \lim_{n \to \infty}\frac{B_0^n(w,a)}{\lambda^n}.$$ Since $B_0$ is primitive, the limit of $B^n_0 / \lambda^n$ is a matrix $P$ that is spanned by a positive eigenvector corresponding to $\lambda$. Since left eigenvector of $B_0$ is also positive, all columns of $P$ are positive multiples of each other. Thus $d_{w,b} = P(w,b)$ is a scalar multiple of $d_{w,a} = P(w,a)$ which does not depend on $w$. We call this constant $\kappa_1$.  

Now consider the case when length of $w$ is $l$. We will first show that the constant $\kappa_l$, where $d_{w,b} = \kappa_l d_{w,a}$, does not depend on $w$ and then we will show that $\kappa_l = \kappa_1$ for all $l\geq 2$. Since $\lambda$ is the unique top eigenvalue of $\mathcal{B}_l$, $\lim_{n \to \infty}\mathcal{B}_l^n / \lambda^n$ is a matrix $P$ whose column span is an eigenvector corresponding to $\lambda$. Thus $d_{w,b} = P(w,b)$ is a scalar multiple of $d_{w,a} = P(w,a)$ which does not depend on $w$. We call this constant $\kappa_l$. 

Now we will show that $\kappa_l = \kappa_1$. Let $w$ be a word of length one. We have $d_{w,b} = \sum_{e \in \mathbb{A}} d_{we,b}$. Also $d_{w,b} = \kappa_1 d_{w,a}$ and $d_{we,b} = \kappa_2 d_{we,a}$. Thus we have $\kappa_1 d_{w,a} = \kappa_2 \sum_{e \in \mathbb{A}}d_{we,a} = \kappa_2 d_{w,a}$ which implies $\kappa_2 = \kappa_1$. We can repeat the same argument to get $\kappa_l = \kappa_1$ for every $l\geq 2$.  
\end{proof}

To summarize the results about substitutions we have the following proposition. 
\begin{proposition}\label{P:MainResultSubstitution}
Let $\zeta$ be a substitution on an alphabet $\mathbb{A}$ such that the transition matrix is lower triangular block diagonal with top left block $B_0$ primitive, and for every $e \in B_0$, $\zeta(e)$ starts and ends with a letter in $B_0$. Then there is a fixed infinite word $\rho$ obtained by iterating a letter in $B_0$ under $\zeta$. Moreover, the frequency of a word $w$ on $\mathbb{A}$ in $\rho$ that crosses $B_0$ is well defined up to scale and satisfies Kirchhoff's law. 
\end{proposition}
\subsubsection{Train track map as a substitution}
Let $\oo$ be a free group outer automorphism. Let $\phi: G \to G$ be a completely split train track representative of $\oo$ with filtration $\fltr$. The transition matrix for $\phi$, denoted $M_{\phi}$, is lower triangular block diagonal. Let $a$ be an edge in an EG stratum $H_r$ such that up to taking powers $\phi(a)$ starts with $a$. Let $\rho_a = \lim_{n \to \infty}\phi^n(a)$. We want to understand the frequency of occurrence of paths in $G_r$ that cross $H_r$ and appear in $\rho_a$. We may not be able to treat $\phi$ as a substitution directly since there could be cancellations and inverse of edges would have to be treated separately. The proof of the next proposition explains how to view a completely split train track map as a substitution for the purpose of calculating frequencies of certain paths. 

We set up some notation about exceptional paths that will be used in the next proposition. Let $e_1, e_2 \in G$ be two linear edges such that $\phi(e_1) = e_1 \sigma^{d_1}$ and $\phi(e_2)=e_2 \sigma^{d_2}$ where $\sigma$ is an INP and $d_1 \neq d_2$. If $d_1, d_2 >0$, then $x_m = e_1 \sigma^m \overline{e}_2$ where $m \in \mathbb{Z}$ is an exceptional path. We say $x_m$ has \emph{width} $|m|$. Let $\delta = d_1 - d_2$. Then $\phi(x_m)$ is the exceptional path $x_{m+\delta}$. 

\begin{proposition}\label{P:TrainTrackSubstitution}
Let $\phi: G \to G$ be a completely split train track map. Let $a$ be an edge in an EG stratum $H_r$ such that $\phi(a)$ starts with $a$, and let $\rho_a := \lim_{n \to \infty}\phi^n(a)$. Let $\gamma$ be a path in $G_r$ that crosses $H_r$. Then $$\lim_{n \to \infty} \frac{(\gamma, \phi^n(a))}{\lambda^n} =:d_{\gamma,a}$$ exists and is non-negative. Here $\lambda$ is the Perron-Frobenius eigenvalue of the aperiodic EG stratum $H_r$. If $b \in H_r$ is another edge, then for every $\gamma$ as above, $$d_{\gamma,b} = \kappa d_{\gamma,a}$$ where $\kappa$ is a constant with $\kappa = \kappa(a, b, \phi|_{H_r})$.
\end{proposition}
\begin{proof}
We will first show how to obtain a substitution from the completely split train track map $\phi$. Then applying Proposition~\ref{P:MainResultSubstitution} to this substitution concludes the proof. The ray $\rho_a$ is completely split and the terms of the complete splitting, called splitting units, of $\rho_a$ form an alphabet $\mathbb{A}_{\infty}$ for a substitution. But $\mathbb{A}_{\infty}$ can be infinite if there are exceptional paths. We will define a finite alphabet $\mathbb{A}_{\gamma}$, which depends on $\gamma$, by identifying some elements in $\mathbb{A}_{\infty}$ in order to calculate the frequency of occurrence of $\gamma$ in $\rho_a$. We will also show that the frequency of $\gamma$ in $\rho_a$ does not depend on the choice of the alphabet $\mathbb{A}_{\gamma}$. Let $\mathcal{N}$ be the set of all INPs, $r$-taken connecting paths and exceptional paths that appear in $\rho_a$.

Before we define the alphabet $\mathbb{A}_{\gamma}$, we define a relation from the set of all finite paths in $\rho_a$ that cross $H_r$, denoted $\mathcal{P}_r(\rho_a)$, to the set of all finite words on $\mathbb{A}_{\infty}$, denoted $\mathcal{W}(\mathcal{A}_{\infty})$, $$r: \mathcal{P}_r(\rho_a) \to \mathcal{W}(\mathcal{A}_{\infty}).$$
For a finite path $\gamma \in \mathcal{P}_r(\rho_a)$, the set $r(\gamma)$ consists of the following words: 

\begin{enumerate}[(a)]
\item If an occurrence of $\gamma$ in $\rho_a$ is a  concatenation of splitting units, then $r(\gamma)$ contains the corresponding word on $\mathbb{A}_{\infty}$. 

\item If an occurrence of $\gamma$ in $\rho_a$ is a subword of an INP $\sigma$, then $r(\gamma)$ contains the element of $\mathbb{A}_{\infty}$ determined by $\sigma$, denoted $w_{\sigma}$. There are only finitely many INPs that appear in $\rho_a$ therefore the number of occurrences of a path $\gamma$ in an INP is bounded. If $\sigma$ contains $n$ occurrences of $\gamma$, then we let $r(\gamma)$ contain $n$ copies of $w_{\sigma}$. Note that a path $\gamma$ in $\mathcal{P}_r(\rho_a)$ is not contained in an exceptional path or an $r$-taken connected path. 

\item If an occurrence of $\gamma$ has partial overlaps with some elements of $\mathcal{N}$, then consider a path $\gamma^{\p}$ such that $\gamma^{\p}$ is the smallest subpath of $\rho_a$ that is a concatenation of splitting units and which contains $\gamma$. Then $r(\gamma)$ contains the word on $\mathbb{A}_{\infty}$ corresponding to $\gamma^{\p}$.
\end{enumerate}

Thus every occurrence of $\gamma$ in $\rho_a$ corresponds to the occurrence of some word from $r(\gamma)$ in $\rho_a$. Note that $r(\gamma)$ can be an infinite set, for instance, when $\gamma$ has partial overlap with infinitely many exceptional paths in $\rho_a$. But the set of words in $r(\gamma)$ viewed in the alphabet $\mathbb{A}_{\gamma}$, defined below, will form a finite set. We now define the alphabet $\mathbb{A}_{\gamma}$. For simplicity, let's assume that $\gamma$ intersects only one family of exceptional paths, say determined by linear edges $e_1, e_2 \in G$. 
\begin{itemize}
\item Let $\mathcal{H} = \{ H_r = H_{i_1}, \ldots, H_{i_k}\}$ be the collection of strata crossed by edges in $H_r$. For every $H_{i_j}$, let $\mathbb{A}(H_{i_j})$ be the alphabet which contains an edge and its inverse as distinct letters if they both appear in $\rho_a$ otherwise the edge with the orientation that appears.

An edge in $G$ is called a \emph{Type 1} edge if it always appears with positive or negative orientation but not both in $\rho_a$. An edge which appears with both orientations in $\rho_a$ is said to be of \emph{Type 2}. If $H_{i_j}$ is an EG stratum, then either all edges in $H_{i_j}$ are Type 1 or all are Type 2 (see \cite{U:HyperbolicIWIP} for proof). Thus a substitution on $\mathbb{A}(H_r)$ representing $\phi$ restricted to $H_r$ is primitive. 


\item Now consider splitting units which are INPs, $r$-taken connecting paths and exceptional paths. Let $\mathbb{A}(\mathcal{N}_{\gamma})$ be an alphabet defined as follows:  
\begin{enumerate}[(a)]
\item All oriented INPs and $r$-taken connecting paths that appear in $\rho_a$ are contained $\mathbb{A}(\mathcal{N}_{\gamma})$. There can be infinitely many INPs in $G_r$ but only finitely many appear in $\rho_a$. 
\item Suppose $\gamma$ contains an exceptional path determined by $e_1, e_2$  or a subsegment of an exceptional path determined by $e_1, e_2$. Let $N$ be the maximum length of such an exceptional path that appears in $\gamma$, in $\phi(e)$ for all edges $e$ in $H_r$ and in an $r$-taken connecting path. Then $\mathbb{A}(\mathcal{N}_{\gamma})$ contains exceptional paths determined by $e_1, e_2$ of width less than equal to $N+1$ as distinct elements. All other exceptional paths determined by $e_1, e_2$ of width greater than $N+1$ correspond to a single element of $\mathbb{A}(\mathcal{N}_{\gamma})$.  
\item Suppose $\gamma$ does not intersect an exceptional path determined by $e_1, e_2$. Then all exceptional paths determined by $e_1, e_2$ correspond to a single element of $\mathbb{A}(\mathcal{N}_{\gamma})$. 
\end{enumerate}

\item Let $\mathbb{A}_{\gamma}$ be defined as the set $\mathbb{A}(H_{i_1}) \sqcup \cdots \sqcup \mathbb{A}(H_{i_k}) \sqcup \mathbb{A}(\mathcal{N}_{\gamma})$ and let $\zeta_{\gamma, \phi}$ be a substitution on $\mathbb{A}_{\gamma}$ determined by $\phi$. Let $\tilde{r}(\gamma)$ be the set of words in $r(\gamma)$ viewed in the alphabet $\mathbb{A}_{\gamma}$. Then $\tilde{r}(\gamma)$ is a finite set of words on $\mathbb{A}_{\gamma}$. The frequency of occurrence of a path $\gamma \in \mathcal{P}_r(\rho_a)$ in $\rho_a$ is given by the sum of the frequencies of the words in $\tilde{r}(\gamma)$.  \end{itemize}

If we replace $N+1$ by $N+C$ for any $C\geq 1$ in the above construction to get a different alphabet $\mathbb{A}_{\gamma}^{\p}$, then the frequency of $\gamma$ calculated from the two alphabets is the same. More precisely, let $\mathbb{A}_{\gamma}$ and $\mathbb{A}_{\gamma}^{\p}$ be two alphabets which differ only in the naming of exceptional paths determined by $e_1, e_2$ of length greater than $N+1$. Let $\zeta$ and $\zeta^{\p}$ be the corresponding substitutions, and let $\tilde{r}(\gamma)$ and $\tilde{r}^{\p}(\gamma)$ be the set of words in $r(\gamma)$ viewed in $\mathbb{A}_{\gamma}$ and $\mathbb{A}_{\gamma}^{\p}$ respectively. An exceptional path maps to another exceptional path under $\phi$. Therefore $\zeta$ and $\zeta^{\p}$ have the same growth rate when restricted to $\mathbb{A}(H_r)$. Since the number of occurrences of $\gamma$ does not change, we get that the two substitutions yield the same frequency for words in $\tilde{r}(\gamma)$ and $\tilde{r}^{\p}(\gamma)$ and hence the same frequency for $\gamma$.

Thus, we have obtained an alphabet $\mathbb{A}_{\gamma}$. The completely split train track map $\phi$ induces a substitution $\zeta_{\gamma}$ on this alphabet. Now Proposition~\ref{P:MainResultSubstitution} can be applied to $\zeta_{\gamma}$ to compute the frequency of occurrence of $\gamma$ in $\rho_a$. Different substitutions constructed here for different words $\gamma$ differ only in exceptional paths. Since an exceptional path maps to another exceptional path these different substitutions have the same growth rate when restricted to $\mathbb{A}(H_r)$. Also Kirchhoff's law still holds for frequencies of paths in $\rho_a$ because $(\gamma, \phi^n(a))$ and $\sum_{e \in G_r}(\gamma e, \phi^n(a))$ differ by a bounded amount.
\end{proof}

We do some examples below to exhibit how to view a completely split train track map as a substitution. 
\begin{example}
Let $R_3$ be the rose on three petals with labels $a,b,c$. Consider a homotopy equivalence $\phi: R_3 \to R_3$ given by $$\phi(a) = a, \phi(b) = Bac, \phi(c) = CBac.$$ Here capital letters denote inverses. The transition matrix for $\phi$ is 
\begin{center}
\begin{tabular}{c}
$ b \quad c \quad a$\\
$\begin{bmatrix}1 & 2 & 0 \\ 1 & 1 & 0 \\ 1 & 1 & 1\end{bmatrix}$
\end{tabular}
\end{center}
There are two strata $H_1 = \{a\}$ and $H_2 = \{b,c\}$. Every edge in $H_2$ is of Type 2. Let $\rho_C = \lim_{n \to \infty}\phi^n(C)$. We have $\mathcal{H} = \{H_2, H_1\}, \mathbb{A}(H_2) = \{b,c,B,C\}$ and $\mathbb{A}(H_1) = \{a, A\}$. Since there are no exceptional paths, we use one alphabet $\mathbb{A} = \{b,c,B,C, a, A\}$ and a substitution $\zeta_{\phi}$ on $\mathbb{A}$ whose transition matrix is given by

\begin{center}
\begin{tabular}{c}
$ b \enskip \,\, c \enskip \,\,B \enskip \,\, C \enskip \,\,a \enskip \,\,A$\\
$\left[\begin{array}{cccccc}
0&0&1&1&0&0 \\ 1&1&0&1&0&0 \\ 1&1&0&0&0&0 \\ 0&1&1&1&0&0 \\   1&1&0&0&1&0\\ 0&0&1&1&0&1
\end{array} \right]$
\end{tabular}
\end{center}
\end{example}
\begin{example}\label{E:TT as Substitution}
Consider a homotopy equivalence $\phi: R_5 \to R_5$ given by 
$$\phi(a) = ab,\phi(b) = bab, \phi(c) = cae, \phi(d) = dc\sigma d, \phi(e) = dcae$$ 
where $\sigma = abAB$ is a Nielsen path. There are two strata $H_1 = \{a,b\}$ and $H_2 = \{c,d,e\}$. Let $\rho_c = \lim_{n \to \infty}\phi^n(c)$. We have $\mathcal{H} = \{H_2, H_1\}, \mathbb{A}(H_2) = \{c,d,e\}, \mathbb{A}(H_1) = \{a, b\}$ and $\mathbb{A}(\mathcal{N})=\{\sigma\}$. Since there are no exceptional paths, we use one alphabet $\mathbb{A} = \{c,d,e,a,b,\sigma\}$ and a substitution $\zeta_{\phi}$ on $\mathbb{A}$ whose transition matrix is given by

\begin{center}
\begin{tabular}{c}
$ c \quad d \quad e \quad a \quad b \quad \sigma$\\
$\left[\begin{array}{cccccc}
   1&1&1&0&0&0 \\ 0&2&1&0&0&0 \\ 1&0&1&0&0&0 \\ 1&0&1&1&1&0 \\  0&0&0&1&2&0\\ 0&1&0&0&0&1
  \end{array} \right]$
\end{tabular}
\end{center}

In this example, the frequency of occurrence of the edge path $ca$ in $\rho_c$ comes from the occurrence of the words $ca$ and $c\sigma$ in $\rho_c(\zeta_{\phi})$. Thus the frequency of $ca$ in $\rho_c$ is equal to $d_{ca, c}+d_{c\sigma,c}.$ 
\end{example}

\begin{example}
This example illustrates the discussion of exceptional paths in Proposition~\ref{P:TrainTrackSubstitution}. Consider a homotopy equivalence $\phi: R_6 \to R_6$ given by 
\begin{center}\begin{tabular}{l l}
$\phi(a) = ab,$ & $\phi(b) = bab,$ \\ $\phi(c) = c\sigma^2,$ & $\phi(d) = d\sigma,$ \\ $\phi(e) = eaf,$ & $\phi(f) = fc\sigma D eaf$, 
\end{tabular}
\end{center}
where $\sigma = abAB$. Some exceptional paths are $x_i = c\sigma^i D$ for $i >0$. 
To calculate the frequency of words like $fx_4$ or $fc\sigma^4$ in $\rho_f$, consider the alphabet $\mathbb{A} = \{e, f, a, b, c, D, x_1, x_2, x_3, x_4, x_5, \sigma, \overline{\sigma}\}$ and substitution $\zeta$ such that 
\begin{center}\begin{tabular}{l l}
$\zeta(a) = ab,$ & $\zeta(b) = bab,$ \\
$\zeta(c) = c\sigma^2$, & $\zeta(d) = d\sigma$, \\
$\zeta(f) = fx_1eaf,$ & $\zeta(e) = eaf,$ \\

$\zeta(\sigma)=\sigma,$ &  $\zeta(\overline{\sigma}) = \overline{\sigma},$\\
$\zeta(x_i) = x_{i+1}$ & for $1\leq i \leq 3$ \\ $\zeta(x_4) = \zeta(x_5) = x_5$, & \\  
\end{tabular}
\end{center}
The path $\gamma = fc\sigma^4$ does not occur as a concatenation of splitting units in $\rho_f$. The path $\gamma^{\p} = fx_4$ is the smallest subpath of $\rho_f$ that is a concatenation of splitting units and contains $\gamma$. Thus the frequency of occurrence of $\gamma$ is the same as the frequency of occurrence of $\gamma^{\p}$. 
\end{example}
\subsection{Example of an outer automorphism relative to $\ffa$ when $\op{rank}(\ffa)=\op{rank}(\free)$} This is an example of a relative fully irreducible outer automorphism when rank of cofactor of $\ffa$ is zero. 
Let $\free = \la a, b, c \ra$ and let $\ffa = \{[\la a \ra],[\la b \ra],[\la c \ra]\}$. Let $\oo$ be an outer automorphism given by $$\oo(a) = a, \oo(b) = aCbcA, \oo(c) = CbcBc.$$ Let $\phi: G \to G$ be a relative train track representative of $\oo$ with $G$ as in Figure~\ref{F:TrainTrack}.
\begin{figure}[h]
    \centering
\includegraphics{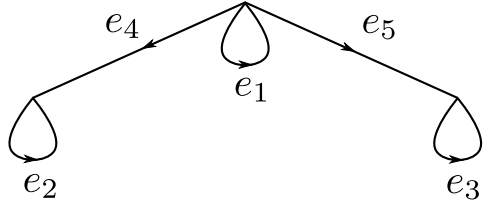}
    \caption{The graph $G$}
    \label{F:TrainTrack}
\end{figure}

The marking given by $$a\to e_1, b\to e_1e_4e_2E_4E_1, c\to e_5e_3E_5.$$ The map $\phi$ is given by 
\begin{center}
\begin{tabular}{l l l}
$\phi(e_1) = e_1$ & $\phi(e_2) = e_2$ & $\phi(e_3) = e_3$ \\ $\phi(e_4) = e_5E_3E_5e_1e_4$ & $\phi(e_5) = e_5E_3E_5e_1e_4e_2E_4E_1e_5 $ &
\end{tabular}
\end{center}
and the transition matrix for $\phi$ is
\begin{center}
\begin{tabular}{c}
$e_5 \,\,\, e_4 \,\,\, e_3 \,\,\, e_2 \,\,\, e_1 $\\
$\begin{bmatrix}  3 & 2 & 0 & 0 & 0 \\ 2 & 1 & 0 & 0 & 0 \\ 1 & 1 & 1 & 0 & 0 \\ 1 & 0 & 0 & 1 & 0 \\ 2 & 1 & 0 & 0 & 1 \end{bmatrix}.$
\end{tabular}
\end{center}
\end{example}

\bibliographystyle{alpha}
\bibliography{bib2}

\end{document}